\documentclass[11pt]{article}
\usepackage[width=16.5cm, height=23cm]{geometry}

\usepackage{graphicx}
\usepackage{amssymb}
\usepackage{amsmath}
\usepackage{amsthm}
\usepackage{dsfont}
\usepackage{xcolor}
\usepackage{epstopdf}

\def\eps{\varepsilon}
\def\R{\mathbb R}
\def\d{\mathrm{d}}
\def\I{1 \!  \!{\mathrm{I}}}
\def\sech{{\, \mathrm{sech} \, }}

\def\edit#1{\textcolor{black}{#1}}
\def\tn#1{\textnormal{#1}}

\newtheorem{theorem}{Theorem}[section]
\newtheorem{lemma}{Lemma}[section]
\newtheorem{hypothesis}{Hypothesis}[section]
\newtheorem{remark}{Remark}[section]

\begin{document}

\title{\edit{Local stable and unstable manifolds} and their control in nonautonomous finite-time flows}

\author{Sanjeeva Balasuriya \\
{\small School of Mathematical Sciences, University of Adelaide, SA 5005, Australia}}

\date{\textcolor{red}{Under consideration for publication in {\em Journal of Nonlinear Science}, February 2016}
}

\maketitle


\begin{abstract}
\edit{It is well-known that stable and unstable manifolds strongly influence fluid motion in unsteady flows.  These emanate
from hyperbolic trajectories, with the structures moving nonautonomously in time.  The local directions of emanation at each instance in time is the focus of this article. Within a nearly autonomous
setting, it is shown that these time-varying directions can be characterised through the accumulated effect of velocity shear. 
Connections to Oseledets spaces and projection operators in exponential dichotomies are established.}  Availability of data for both infinite and finite time-intervals is considered. 
With microfluidic flow control in mind, a methodology for manipulating these directions in any prescribed time-varying fashion by applying a local velocity shear is developed.   The results are verified for both smoothly and discontinuously time-varying directions using finite-time Lyapunov exponent fields, and excellent agreement is obtained.
\end{abstract}



\section{Introduction}
\label{sec:intro}

Transport in nonautonomous (unsteady) flows such as in oceanic and atmospheric jets and eddies/vortices/rings, or microfluidic 
channels with several fluids, is well-known to be strongly influenced by dominant flow structures and separators.
Examples include the invisible flow barrier off the west coast of Florida which protected the coast during the Deepwater
Horizon oil spill, the Antarctic polar vortex, and the interface between a microdroplet and a carrier fluid in a microchannel.
These structures all move with time, and the ability to demarcate, manipulate, or control the flow across, such structures
has profound implications at a range of scales from the geophysical to the nanofluidic.  Over the last decade or two, the
connection between such flow barriers and the dynamical systems concepts of stable and unstable manifolds has become
well-established \cite{halleryuan,hallerreview,peacockhaller,peacockdabiri,microfluid_review}.  The results of this article are motivated
by the desire to control these flow separators in a way that one chooses, with the idea of being able to govern the transport
occurring in microfluidic devices.  The fact that the flow is of low Reynolds number at microscales means that fluid mixing
is suppressed; yet in many devices the wish is to improve mixing.  By manipulating the flow separators, i.e., the stable and
unstable manifolds, one can focus directly on the transport templates which regulate fluid mixing.

In autonomous (steady) flows defined for infinite times, the concepts of saddle-like fixed points and their eigenvalues and eigenvectors
enable the determination of stable and unstable manifolds locally.  The {\em global} structure of these invariant manifolds
form a template which governs transport in space, yet their {\em local} structure (and indeed global definitions in terms of
exponential decays) are associated directly with the eigenvalues and eigenvectors.   In particular, the directions in which these emanate locally from the saddle point are defined by the eigenvectors.  Unfortunately, in {\em nonautonomous} flows defined
only for {\em finite} times---i.e., for any realistic situation based for example on observational or experimental 
data---{\em instantaneously computed} fixed points, eigenvalues or eigenvectors have no relevance to transport.  

There are a variety of diagnostic techniques which purportedly identify transport templates in nonautonmous flows which are
defined over finite times.
There are variously defined as entities of extremal attraction/repulsion \cite{farazmandhaller,blazevskihaller,teramoto}, minimal instantaneous flux/deformation \cite{halleryuan,minimumflux},
extremal curve/surface length/area deformation \cite{hallervariational,hallerberonvera} or extremal curvature deformation \cite{mabollt}, or ridges extracted
from finite-time Lyapunov exponent (FTLE) fields \cite{peacockdabiri,peacockhaller,chian,green,shadden,karraschhaller,kelley,branickiwiggins2010} or transfer (Perron-Frobenius) operator singular vector 
fields \cite{froylandpadberg2014,froyland2013,froylandfinite2010,froylandpadberg2009}.   Alternative methods identify these using topological braid theory \cite{allshousethiffeault}, ergodic theory \cite{budisicmezic}
or various averages along trajectories \cite{pojehallermezic,mancho_2013,mancho_2003,mezicscience}.   
Each of these diverse methods offer computational tools for extracting precisely its own definition;
there are few theoretically established connections between them \cite{hallerreview}.
 The {\em hyperbolic}-type
transport templates  can be construed as {\em boundaries} between
coherent structures, and are therefore of great importance in any transport characterisation.
These are characterised in Table~\ref{table:analogues}, and 
in all these situations, one focus is essentially determining the third column of Table~\ref{table:analogues}, which
will be different for each diagnostic method.
For the specific case of FTLEs (noting that FTLEs are not always valid \cite{branickiwiggins2009,karrasch,norgardbremer,schindler,hallerreview}), this column would arguably have
(a) intersections of forward- and backward-time FTLE ridges\footnote{Not all such intersections correspond to 
hyperbolic trajectories, since \edit{hetero}clinic trajectories are also associated with such intersections.}, 
(b) FTLE values, and (c) tangents to FTLE ridges.  

\begin{table}[t]
\centering
\begin{tabular}{| c | c | c |} \hline
{\bf {\small Autonomous, infinite-time}} & {\bf {\small Nonautonomous, infinite-time}} & {\bf {\small Nonautonomous, finite-time}} \\ \hline 
Saddle fixed points & Hyperbolic trajectories & (a) \\ \hline
Eigenvalues & Exponential decay rates & (b) \\ \hline
Eigenvectors & Tangents to stable/unstable manifolds & (c) \\ \hline
\end{tabular}
\caption{Analogues of important flow entities.}
\label{table:analogues}
\end{table}

Even for nonautonomous infinite-time flows, the definitions in the second column of Table~\ref{table:analogues} are
not easy to apply for detection of these coherent entities.  The time-varying hyperbolic trajectories and their stable and unstable manifolds are defined
{\em implicitly}, using the concept of exponential dichotomies \cite{coppel,palmer,battellilazzari,unsteady}, which is difficult to use
computationally.   Methods which would enable easier characterisation and control of these entities would
be of value---and in particular, those generalisable to finite-time flows.  In this spirit, this
article develops techniques which are relevant when the flow is nearly autonomous, that is, when the nonautonomous
component appears as a perturbation on an autonomous (steady) flow.  A frequently considered assumption in 
oceanographic and other flows \cite{gaultier,miller,dovidio2004},  this form of model has additional appeal since it can be considered a particular
realisation of a stochastic perturbation, a theme which is attracting considerable recent attention \cite{guckenheimer,lamb,nguyenthulam}.  From the microfluidic perspective, this assumption has value when the idea is to
disturb a steady laminar flow (in which transport is suppressed) by imposing an agitation velocity in order to promote transport
\cite{stroock,whitesides,wangturbulence,droplet,microfluid_review}.

\edit{For nonautonomous infinite-time flows, the tangent vectors to stable and unstable manifolds, computed locally at
the associated hyperbolic trajectory location at each time $ t $, form well-defined entities.  This characterisation moreover
generalises the concept of eigenvectors at saddle points in autonomous flows, since such eigenvectors are certainly also
local tangent vectors to stable/unstable manifolds, as defined in terms of the classical stable/unstable manifold theorem
\cite{guckenheimerholmes,arrowsmithplace}.  Tangent vectors to stable/unstable manifolds are also invariant under
affine coordinate transformations (a feature not shared by eigenvectors), thereby satisfying the concept of `objectivity'
\cite{hallerreview}.}
\edit{Hence, this article focusses specifically on these entities, and develops in Section~\ref{sec:theory} a theory valid for nearly autonomous two-dimensional flows.  Section~\ref{sec:hyperbolic}
characterises the time-varying location of the hyperbolic trajectory (i.e., the anchor point to which the local stable/unstable
manifolds are attached), and Section~\ref{sec:eigenvector} establishes that the directions in which the local stable and unstable manifolds emanate can be expressed as a rotation from the autonomous
eigenvector directions.  It turns out that this rotation is governed by the accumulated effect of the time-varying
local velocity shear. The relationship of the tangent
vectors to the concepts
of Oseledets spaces \cite{oseledets,froyland2013,froylandcocycle2010,liang,ginelli}, and projection operators associated with exponential dichotomies \cite{coppel,palmer,battellilazzari},  is elucidated 
in Section~\ref{sec:theoryconnect}.}  

Section~\ref{sec:finite} addresses the relaxation of the infinite-time results to finite times, 
when entities reminiscent of stable and unstable manifolds appear to be present.  (Defining them using exponential
dichotomies fails since {\em anything} can be imputed to be bounded by an exponentially decaying function over finite
times; various approaches \cite{doan,karrasch,ducsiegmund,berger} have been suggested to tackle this.)
A common approach for finite times is to
consider the flow from a certain fixed initial time to a fixed final time, governed by a specified nonautonomous velocity
field \cite{karrasch,doan,froylandfinite2010,mosovskymeiss,froyland2013,froylandpadberg2009,mabollt}.  The flow map for this time interval can be obtained computationally, for example, by flowing forward many
trajectories, and thereafter transport issues are usually analysed using just this one-step flow map.  Nonautonomy
is sometimes (but not always) implicit in this approach 
by allowing either the initial or final time to vary, but this requires performing calculations for differing flow
maps \cite{farazmandhaller,mezicscience,hallerreview}.   Finite-time reality might be better approximated if assuming that 
the velocity data is available over a fixed finite-time interval $ [-T,T] $
(possibly at discrete values), but that nonautonomous information is required for times $ t \in (-T,T) $, using {\em all} the
available data at each instance in time \cite{farazmandhaller}.  Under this interpretation, in the nearly autonomous situation, 
leading-order expressions for the finite-time versions of
hyperbolic trajectories and local tangent vectors to their associated stable and unstable
manifolds are furnished.  This approach effectively
supplies an alternative---which generalises stable/unstable manifolds in nearly autonomous flows---to 
the third column in Table~\ref{table:analogues}. The development highlights the seldom noticed fact that if data
from a finite-time sample $ t \in [-T,T] $ is used to determine invariant manifold-like entities, velocities {\em outside} 
the interval could affect the locations of the computed entities.  To account for such errors due to lack of data outside
the interval $ [-T,T] $,  an error estimate for the tangent vectors, as a function of both
$ t $ and $ T $, is discussed in (\ref{eq:finitebound}), under the condition of bounded velocity shear.  

\edit{Section~\ref{sec:control} poses the inverse problem: given a desired hyperbolic trajectory location and
local stable/unstable manifold emanation direction, each of which is varying in time in a specified fashion, is it possible
to determine the required velocity conditions to achieve this?  This can be thought of as a control problem, i.e., determining
the control velocity needed to force local stable/unstable manifolds to behave in a required fashion.}  This result nicely
complements already existing results on controlling hyperbolic trajectories \cite{saddle_control,controlnd}, and supports 
the (as yet not easily implementable) first result on controlling stable and unstable manifolds \cite{manifold_control}.  Since
Theorem~\ref{theorem:control} provides the methodology of pointing stable and unstable manifolds in user-desired time-varying
directions, this provides insight into how best to focus energy in the most relevant areas, in order to obtain desired mixing
characteristics.

The theoretical arguments of Section~\ref{sec:theory} are verified in Section~\ref{sec:taylorgreen}, where a nearly autonomous flow is computationally
analysed through the generation of spatially and temporally discrete data over a finite time-interval $ [-T,T] $.  This is
performed via a two-step analysis: first, the theory is used to determine the nonautonomous velocity for $ t \in (-T,T) $ 
needed to have 
hyperbolic trajectories follow user-defined motion, with their stable and unstable manifolds also rotating in a specified
time-varying fashion, and second, the data is used to {\em a posteriori} verify the errors implied by the finite-time
definitions.   Two different types of nonautonomous perturbations are evaluated: a 
time-periodic manifold rotation with fairly large nonautonomous part, and an abruptly changing hyperbolic
trajectory location and manifold rotation.  These examples were chosen to deliberately challenge the expected realm
of viability of the theory, but in both cases, excellent results were obtained when compared with FTLE computations.
The techniques therefore offer substantial promise in controlling directions of emanation of coherent structure 
boundaries in finite-time nonautonomous flows, with additional insights into how these directions are related to the
local velocity shear.

\section{Theoretical framework}
\label{sec:theory}

Consider for $ x \in \Omega \subset \R^2 $ the nonautonomous dynamical system
\begin{equation}
\dot{x} = F(x,t,\eps)
\label{eq:flow}
\end{equation}
where the parameter $ \eps \in [0,\eps_0) $ where $ \eps_0 \ll 1 $.   The vector field $ F $ is assumed to be defined
for $ t \in [-T,T] $ (if $ T < \infty $), with (\ref{eq:flow}) being valid for $ t \in (-T, T) $; 
the $ T = \infty $ situation is the classical infinite-time scenario which shall
be the first focus.  An alternative
representation of the nonautonomous flow of (\ref{eq:flow}) is to consider its augmented system
\begin{equation}
\left. \begin{array}{l l l }
\dot{x} & = & F(x,t,\eps) \\
\dot{t} & = & 1
\end{array} \right\}
\label{eq:flowappend}
\end{equation}
in the appended $ \Omega \times (-T,T) $ phase space.

\begin{hypothesis}
\label{hyp:F}
$ F: \Omega \times (-T,T) \times [0,\eps_0) \rightarrow \R^2 $ is such that 
\begin{itemize}
\item[(a)] $ F \in {\mathrm{C}}_{\mathrm{unif}}^{\edit{3}} \left( \Omega \times (-T,T) \times [0,\eps_0) \right) $;
\item[(b)] The function $ f(x) := F(x,t,0) $ is independent of $ t $;
\item[(c)] There exists $ a \in \Omega $ such that $ f(a) = 0 $;
\item[(d)] $ D f(a) $ has eigenvalues $ \lambda_{\tn{s}} $ and $ \lambda_{\tn{u}} $ with corresponding normalised
eigenvectors $ v_{\tn{s}} $ and $ v_{\tn{u}} $, where $ \lambda_{\tn{s}} < 0 < \lambda_{\tn{u}} $.
\end{itemize}
\end{hypothesis}

Some comments on the above hypotheses are in order.  Hypotheses~\ref{hyp:F}(a) contains basic differentiability
and boundedness conditions on $ F $ in relation to $ (x,t,\eps) $.  Recall that the `unif' subscript on the $ {\mathrm{C}}^2 $ indicates that all derivatives up to second-order are uniformly bounded.  Hypothesis~\ref{hyp:F}(b) states that when $ \eps = 0 $, $ F $ is
autonomous, and can be represented by a function $ f(x) $.
Hypotheses~\ref{hyp:F}(c,d) guarantee the presence of a hyperbolic fixed point $ a $ associated with this $ f $, with stable
and unstable manifolds emanating in the directions $ v_{\tn{s}} $ and $ v_{\tn{u}} $, with associated exponential
decay rates $ \lambda_{\tn{s}} $ and $ \lambda_{\tn{u}} $, respectively.  While the normalised $ v_{\tn{s,u}} $ are only unique up to 
a sign, suppose a particular choice has been made; this effectively chooses one of the two branches of the 
stable/unstable manifold.  When viewed in the appended phase-space of (\ref{eq:flowappend}), the hyperbolic fixed
point $ a $ of (\ref{eq:flow}) transforms to a {\em hyperbolic trajectory} $ (a,t) $.  An important omission from the
hypotheses is that $ F $ be area-preserving; the results will work for compressible as well as incompressible flows.

\subsection{\edit{Hyperbolic trajectories}}
\label{sec:hyperbolic}

What is the analogous entity to $ a $ when $ \eps \ne 0 $?  If using the fixed point characterisation, this might be thought of
as a curve of instantaneous fixed/stagnation
points $ a(t) $ which satisfies $ F(a(t), t, \eps) = 0 $ for any $ t $ and $ \eps $.    While this characterisation
continues to be used in some fields, this is well-known to {\em not} have any significant meaning with regards to fluid
transport: these do not follow the Lagrangian flow, are not associated with stable/unstable manifolds, and this
characterisation depends on the frame of reference.   The governing characteristic of $ a $ of interest is
not that it is a fixed point when $ \eps = 0 $, but that {\em it possesses stable and unstable manifolds}.  
The analogous entity when $ \eps \ne 0 $ is a time-varying trajectory $ a(t) $ which is called a {\em hyperbolic trajectory}, and 
whose definition requires {\em exponential
dichotomy conditions} \cite{coppel,palmer,battellilazzari} (see also Section~\ref{sec:theoryconnect}) 
which are extremely difficult to verify in nonautonomous
situations.  
However, Yi's persistence of integral manifold results \cite{yi,yistability}, building on Hale's results for more restrictive
time-dependence \cite{hale},
imply the persistence of the $ \eps = 0 $ hyperbolic trajectory $ (a,t) $ of
(\ref{eq:flowappend}) as a \edit{nearby} entity $ \left( a(t), t \right) $ when  $ \eps \ne 0 $.  

The results will be phrased in terms of the `nonautonomy'
\begin{equation}
g(x,t, \eps) := F(x,t,\eps) - F(x,t,0) = F(x,t,\eps) - f(x) \, , 
\label{eq:g}
\end{equation}
which represents how much $ F(x,t,\eps) $ differs from the autonomous vector field $ f(x) $.
By Hypothesis~\ref{hyp:F}, there exists a constant $ C $ such that
\begin{equation}
\left| g \left( x, t, \eps \right)  \right|  + \left| D g \left( x, t, \eps \right) \right| \edit{+ \left| D^2 g \left( x, t, \eps \right) \right|}
\le C \eps \quad {\mathrm{for}} \quad (x,t,\eps) \in \Omega \times (-T,T) \times [0,\eps_0) \, ;
\label{eq:gbound}
\end{equation}
i.e., $ g = {\mathcal O}(\eps) $.  \edit{Next, define $ G^\perp $ as the rotation of a vector $ G $ by $ + \pi/2 $, and so}
\begin{equation}
v_{\tn{s}}^\perp := \left( \begin{array}{lr} 0 & -1 \\ 1 & 0 \end{array} \right) v_{\tn{s}} \quad {\mathrm{and}} \quad
v_{\tn{u}}^\perp := \left( \begin{array}{lr} 0 & -1 \\ 1 & 0 \end{array} \right) v_{\tn{u}} \, .
\label{eq:perp}
\end{equation}
Thus, $ g $'s components
in the normal directions to the vectors $ v_{\tn{s,u}} $ can be defined by
\begin{equation}
g_{\tn{s,u}}^{\edit{\dagger}}(x,t, \eps) := g(x,t, \eps)^{\edit{\top}} v_{\tn{s,u}}^\perp \, .
\label{eq:gcomp}
\end{equation}

\begin{theorem}[\edit{Hyperbolic trajectory}] 
\label{theorem:hyperbolic}
Suppose $ T = \infty $.  The hyperbolic trajectory $ a(t) $ of (\ref{eq:flow}) in relation to $ a $ can
be represented by the projections
\begin{equation}
\left[ a(t) - a \right]^{\edit{\top}} v_{\tn{s}}^\perp = -   \, e^{\lambda_{\tn{u}} t} \int_t^\infty e^{- \lambda_{\tn{u}} \tau} g_{\tn{s}}^{\edit{\dagger}} \left( a, \tau, \eps \right)  \, \d \tau + \eps^2 H_{\tn{s}}(t,\eps)
\label{eq:hyperbolics}
\end{equation}
and
\begin{equation}
\left[ a(t) - a \right]^{\edit{\top}} v_{\tn{u}}^\perp = \, e^{\lambda_{\tn{s}} t} \int_{-\infty}^t  e^{ -\lambda_{\tn{s}} \tau} g_{\tn{u}}^{\edit{\dagger}} \left( a, \tau, \eps \right)  \, \d \tau + \eps^2 H_{\tn{u}}(t,\eps) \, ,
\label{eq:hyperbolicu}
\end{equation}
where there exist $ K_{\tn{s,u}} > 0 $ such that $ \left| H_{\tn{s,u}}(t,\eps) \right| + \left| \frac{\partial H_{\tn{s,u}}(t,\eps)}{\partial t} \right| 
\le K_{\tn{s,u}}  $ for all $ t \in \R $.
\end{theorem}

\begin{proof}
This formulation is a slight modification to Theorem~2.10 by Balasuriya \cite{tangential}, to which the reader is 
referred to for the proof for the leading-order expression; the bounding of the error term arises from an argument
similar to that of Theorem~\ref{theorem:eigenvector} which will be shown in detail, and hence will be skipped here.
\end{proof}

 
 \begin{remark} \normalfont
 \label{remark:alpha}
While Theorem~\ref{theorem:hyperbolic} gives the {\em projections} of $ a(t) - a $  \edit{in the directions $ v_{\tn{s,u}}^\perp $}
to 
leading-order, an expression for $ a(t) $ \edit{can} be obtained using elementary trigonometry as \cite{tangential}
\begin{equation}
a(t) = a +   \left[ \alpha_{\tn{u}}(t,\eps) v_{\tn{u}}^\perp  + 
\frac{ \alpha_{\tn{u}}(t,\eps) \left( v_{\tn{u}}^{\edit{\top}}  v_{\tn{s}} \right) - \alpha_{\tn{s}}(t,\eps) }{
v_{\tn{s}}^{\edit{\top}} v_{\tn{u}}^\perp  } \, v_{\tn{u}} 
\right] + {\mathcal O}(\eps^2) \, ,
\label{eq:hyperbolic}
\end{equation}
where
\begin{equation}
\alpha_{\tn{s}}(t,\eps) := - e^{\lambda_{\tn{u}} t} \int_t^\infty e^{- \lambda_{\tn{u}} \tau} g_{\tn{s}}^{\edit{\dagger}} \left( a, \tau, \eps \right)   \, \d \tau
\label{eq:alphas}
\end{equation}
and
\begin{equation}
\alpha_{\tn{u}}(t,\eps) :=  e^{\lambda_{\tn{s}} t} \int_{-\infty}^t  e^{ -\lambda_{\tn{s}} \tau} g_{\tn{u}}^{\edit{\dagger}} \left( a, \tau, \eps \right)  \, \d \tau \, .
\label{eq:alphau}
\end{equation}
 \end{remark}

\subsection{\edit{Local stable and unstable manifold directions}}
\label{sec:eigenvector}

When $ \eps = 0 $, the dynamical system (\ref{eq:flow}) is autonomous, possessing a saddle fixed point $ a $ with 
eigenvalues $ \lambda_{\tn{s}} < 0 $ and $ \lambda_{\tn{u}} > 0 $, and corresponding normalised eigenvectors $ v_{\tn{s}} $ and $ v_{\tn{u}} $.
The nonautonomous finite-time version of the fixed point $ a $ has been addressed, and this section approaches the determination of  the
local tangent vectors to the stable and unstable manifolds at the point $ a(t) $ at each time $ t $. 
It is reasonable to expect the new (time-dependent) tangent vector directions $ v_{\tn{s,u}}(t) $ to be $ {\mathcal O}(\eps) $-close to $ v_{\tn{s,u}} $ in each time-slice $ t $.  

\begin{theorem}[\edit{Local stable/unstable manifold directions}]
\label{theorem:eigenvector}
Suppose $ T = \infty $.  Consider the intersections of the stable and unstable manifolds of $ \left( a(t), t \right) $ in a time-slice $ t 
\in \R $.  The local tangential direction to the stable manifold at the point $ a(t) $ can be obtained by an anticlockwise
rotation of $ v_{\tn{s}} $ by an angle
\begin{equation}
\theta_{\tn{s}}(t) = - \, e^{(\lambda_{\tn{u}}-\lambda_{\tn{s}})t} \int_t^\infty e^{(\lambda_{\tn{s}} - \lambda_{\tn{u}})\tau} v_{\tn{s}}^{\edit{\top}}  \edit{D}  g_{\tn{s}}^{\edit{\dagger}} (a,\tau,\eps)  \, \d \tau  + \eps^2 E_{\tn{s}}(t,\eps) \, ,
\label{eq:slopes}
\end{equation}
where there exists a constant $ K_{\tn{s}} $ such that 
\[
 \left| E_{\tn{s}}(t,\eps) \right| + \left| \frac{\partial E_{\tn{s}}(t,\eps)}{\partial t} \right| \le K_{\tn{s}} \quad {\mathrm{for}} \, (t,\eps)  \in 
 \edit{\R}
 \times [0,\eps_0) \, .
 \]
Similarly, the local tangential direction to the unstable manifold at $ a(t) $
is obtained by rotating $ v_{\tn{u}} $ by an
anticlockwise angle
\begin{equation}
\theta_{\tn{u}}(t)  =   e^{(\lambda_{\tn{s}}-\lambda_{\tn{u}})t} \int_{-\infty}^t e^{(\lambda_{\tn{u}} - \lambda_{\tn{s}})\tau} v_{\tn{u}}^{\edit{\top}}  \edit{D} g_{\tn{u}}^{\edit{\dagger}} (a,\tau,\eps) \, \d \tau + \eps^2 E_{\tn{u}}(t,\eps) \, ,
\label{eq:slopeu}
\end{equation}
where there exists a constant $ K_{\tn{u}} $ such that 
\[
 \left| E_{\tn{u}}(t,\eps) \right| + \left| \frac{\partial E_{\tn{u}}(t,\eps)}{\partial t} \right| \le K_{\tn{u}} \quad {\mathrm{for}} \, (t,\eps) \in 
 \edit{\R} \times [0,\eps_0) \, .
\]
\end{theorem}

\begin{proof}
See Appendix~\ref{app:eigenvector}.
\end{proof}

\begin{figure}[t]
\includegraphics[width=0.5 \textwidth, height=0.3 \textheight]{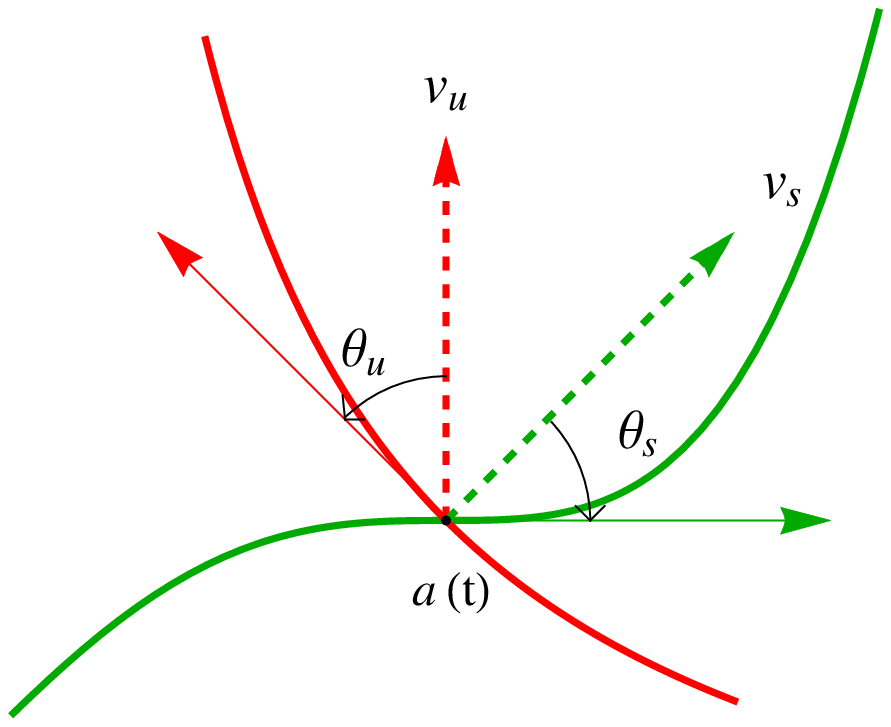}
\includegraphics[scale=0.55]{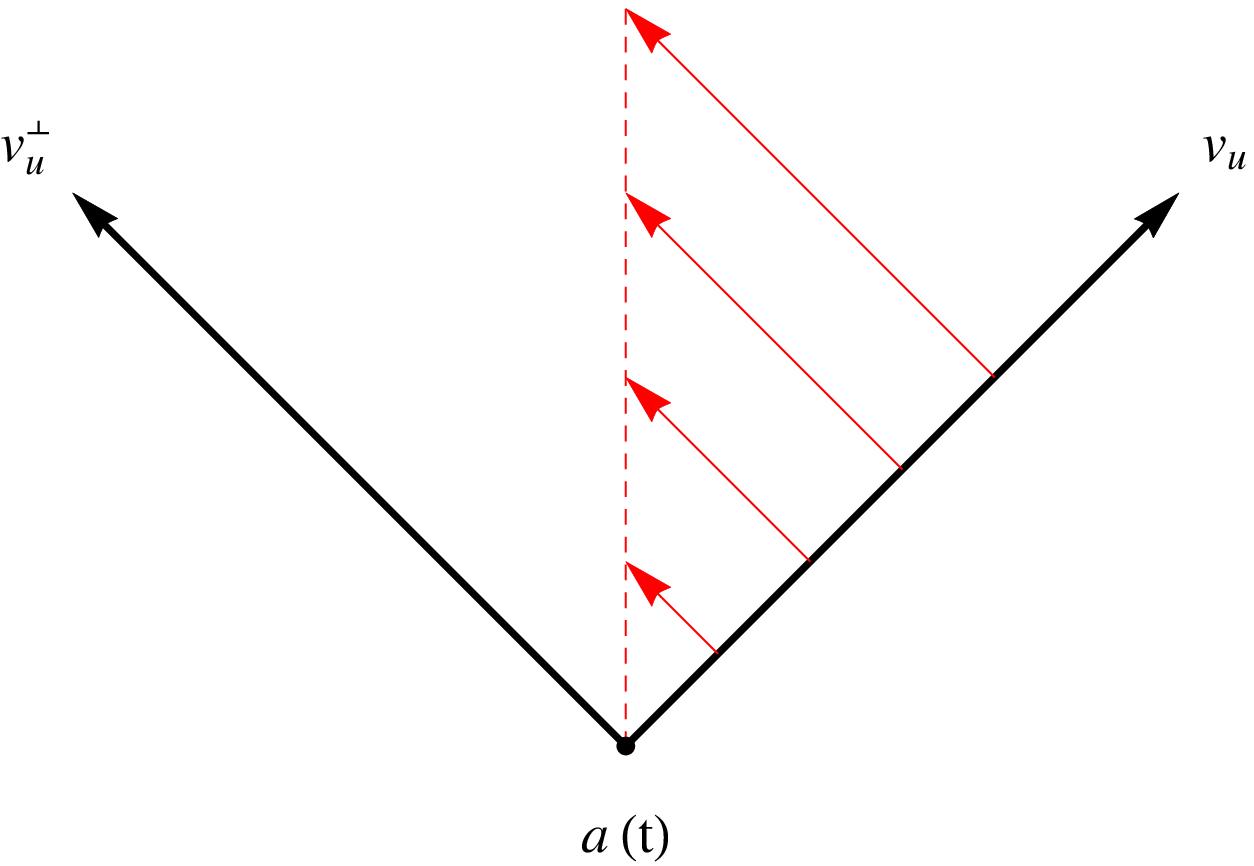}
\caption{(Left) Tangent vectors to stable (green) and unstable (red) manifolds in a time-slice $ t $, as expressed by
(\ref{eq:slopes}) and (\ref{eq:slopeu}).  The dashed arrows are the $ \eps = 0 $ \edit{tangent vector directions}. (Right) Illustration of 
a shear velocity profile tangential to the $ v_{\tn{u}}^\perp $ direction which might intuitively be 
thought to rotate $ v_{\tn{u}} $ in the anticlockwise direction.}
\label{fig:slope}
\end{figure}

The angular rotations are illustrated in the left panel of Fig.~\ref{fig:slope} at a general instance $ t $.  The dashed arrows are
the unperturbed eigendirections associated with $ \eps = 0 $, and the tangent vectors to the stable and unstable 
manifolds are obtained by rotating these anticlockwise by angles $ \theta_{\tn{s,u}}(t) $.  In the situation pictured
in Fig.~\ref{fig:slope}, $ \theta_{\tn{s}}(t) < 0 $ and $ \theta_{\tn{u}}(t) > 0 $.

Additional physical insight into Theorem~\ref{theorem:eigenvector} arises from the observation that the key quantity which is being integrated over all relevant time (either backwards \edit{or} forwards from time $ t $ 
depending on whether the stable or unstable manifold is being considered) involves a term
\begin{equation}
\sigma_{\tn{s,u}}(t,\eps)  := \edit{v_{\tn{s,u}}^\top D g_{\tn{s,u}}^{\edit{\dagger}} (a,\tau,\eps) } =
v_{\tn{s,u}} \cdot \nabla \left[  v_{\tn{s,u}}^\perp  \edit{\cdot g(a,\tau,\eps) } \right]   \, .
\label{eq:shear}
\end{equation}
\edit{The alternative notation using $ \nabla $ highlights that} this is by definition the {\em shear} of the nonautonomous component of the velocity (i.e., proportional to the shear 
strain associated with the velocity
field $ g $) in the $ v_{\tn{s,u}}^\perp $ directions, since it represents the directional derivative of $ g  \cdot v_{\tn{s,u}}^\perp $
in the direction of $ v_{\tn{s,u}} $.   For intuition as to why the shear affects the rotation of the tangent vectors, see 
the right panel of Fig.~\ref{fig:slope} in which a velocity profile for $ g $ is shown in relation to 
the $ v_{\tn{u}} $ and $ v_{\tn{u}}^\perp $ vectors.  In this picture, $ g $ is purely in the $ v_{\tn{u}}^\perp $ direction, and it is
increasing in the coordinate along the $ v_{\tn{u}} $ direction.  Thus $ \sigma_{\tn{u}} $ is positive.  
However, the velocity situation in Fig.~\ref{fig:slope} intuitively
will push the further parts of the vector $ v_{\tn{u}} $  more than parts near $ a(t) $, and thus the vector $ v_{\tn{u}} $ would be
expected to rotate in the anticlockwise direction.   This is the positive direction of rotation; positive $ \sigma_{\tn{u}} $ corresponds
to positive $ \theta_{\tn{u}} $, as is clear from (\ref{eq:slopeu}).  A similar intuition on why the shear rotates $ v_{\tn{s}} $ is
possible.

\subsection{\edit{Connections to alternative characterisations}}
\label{sec:theoryconnect}

Theorem~\ref{theorem:eigenvector} in association with Theorem~\ref{theorem:hyperbolic} enables an intuitive 
geometric characterisation of the exponential dichotomy conditions \cite{coppel,palmer,battellilazzari} 
related to the nonautonomous flow (\ref{eq:flow}).  First consider
$ \eps = 0 $, when $ a $ is a hyperbolic trajectory.  
Let $ Y(t) $ be a fundamental matrix solution to the linearised flow around $ a $, i.e.,
\begin{equation}
 \dot{y} = D f(a) y \, ,
 \label{eq:steadyvariation}
 \end{equation}
and for convenience choose $ Y(0) = \I $, the identity.  Exponential dichotomies \cite{coppel,palmer,battellilazzari}
state that in this situation there is a projection $ P $ and constants $ K_{\tn{s,u}} $ such that
\begin{equation}
\begin{array}{ll} 
\left\| Y(s) P Y^{-1}(t) \right\| \le K_{\tn{s}} e^{\lambda_{\tn{s}} (s-t)} & ~{\mathrm{if}}~ s \ge t \, ,  \\
\left\| Y(s) (\I - P) Y^{-1}(t) \right\| \le K_{\tn{u}} e^{\lambda_{\tn{u}} (s-t)} & ~{\mathrm{if}}~s \le t \, .
\end{array}
\label{eq:exponentialdichotomy}
\end{equation}
Now, it is easily verified that a solution to (\ref{eq:steadyvariation}) which takes the value $ y(t) $ at time $ t $ can be stated as 
$ y(t)  = Y(t) y(0) $ in terms of the fundamental matrix solution.  \edit{I}f $ y(0) $ is
in the range of $ P $, then one can easily use the first exponential dichotomy condition to show that
\[
\left| y(s) \right| \le K_{\tn{s}} e^{\lambda_{\tn{s}} s} \left| y(0) \right| \quad {\mathrm{for}}~ s \ge 0 \, ;
\]
see for example Appendix~A in \cite{unsteady}.  This means that if $ y(0) $ were chosen in $ {\mathcal R} (P) $, the range
of $ P $, the
subsequent solution will decay exponentially with rate $ \lambda_{\tn{s}} $.  This solution followed in time will take the
form $ y(t) = Y(t) P w $, which will therefore lie on the stable fibre, and thus the tangent vector space to the stable manifold at a
general time $ t $ is given by $ {\mathcal R} \left( Y(t) P \right) $.  
 This ostensibly 
depends on time, but for the autonomous equation (\ref{eq:steadyvariation}), it actually turns out to not.
This is since if $ y(0) = k v_{\tn{s}} $ for any $ k $, then $ y(t) = k e^{\lambda_{\tn{s}} t} v_{\tn{s}} $, which continues to be in the same direction
as $ v_{\tn{s}} $.  Specifically, $ P $ in this case can be defined through $ P w = \left( v_{\tn{s}}^\top w \right) v_{\tn{s}} $, and so
$ {\mathcal R} \left( Y(t) P \right) = {\mathcal R} \left( v_{\tn{s}} \right) $.
 Similarly, $ {\mathcal R} \left( Y(t) \left[ \I - P \right] \right) = {\mathcal R} (v_{\tn{u}}) $, the
unstable subspace.
Now when $ \eps \ne 0 $, the relevant linearised equation around the hyperbolic trajectory would be 
\begin{equation}
\dot{y} = D F \left( a(t), t \right) y \, ,
\label{eq:unsteadyvariation}
\end{equation}  
which is now nonautonomous.
An exponential dichotomy condition such as (\ref{eq:exponentialdichotomy}) must be satisfied for the
fundamental matrix $ \tilde{Y}(t) $ of (\ref{eq:unsteadyvariation}), but with different
(but $ {\mathcal O}(\eps) $-close) constants $ \tilde{K}_{\tn{u,s}} $ and $ \tilde{\lambda}_{\tn{s,u}} $ and projection $ \tilde{P} $.
These are difficult to determine in general.
The stable fibres at a general time $ t $ are then given by  $  {\mathcal R} \left( \tilde{Y}(t) \tilde{P} \right) $ which
now generically depends on $ t $; this represents exactly $ {\mathcal R} \left( v_{\tn{s}}(t) \right) $.  The connection to
the work in this article is that
a unit vector of $ {\mathcal R} \left( \tilde{Y}(t) \tilde{P} \right) $ can be obtained from a unit vector $ v_{\tn{s}} $ of
$ {\mathcal R} \left( Y(t) P \right) $ by rotating by $ \theta_{\tn{s}}(t) $, to leading-order in $ \eps $.  A similar statement
holds for the unstable fibres.

The tangent vectors computed in Theorem~\ref{theorem:eigenvector} also have a strong connection to
Oseledets spaces.  The time-variation of the tangent vectors indicates the
$ t $-parametrisation of the basis vectors of the 
stable and unstable Oseledets spaces \cite{oseledets,froyland2013,froylandcocycle2010,liang,ginelli} 
associated with the variational equation (\ref{eq:unsteadyvariation}) evaluated at the hyperbolic trajectory.  That is,
the Oseledets splitting  associated with this specific trajectory at time $ t $ is  $  
{\mathcal R} \left( \tilde{Y}(t) \tilde{P} \right) \oplus {\mathcal R} \left( \tilde{Y}(t) \left[ \I - 
\tilde{P} \right] \right) $, and this is furnished to leading-order by Theorem~\ref{theorem:eigenvector}.

\edit{The tangent vector directions encapsulate the directions in which, if initial conditions to the variational
equation are chosen in that direction, maximal attraction/repulsion occurs.  In this sense, these directions are intuitively
what one might obtain if using standard FTLE methods, but if the directions of maximality associated with the Lyapunov
exponent is also recorded.  However, the methods of this article only apply to FTLE computations/exponential dichotomies/Oseledets spaces at the hyperbolic trajectory
location, and the theoretical connection is valid in infinite-time, since that is required for the definition of these entities.
For finite times, the connections are not as straightforward.  The next section defines and analyses one way in which such
a connection can be made.}

\subsection{Application to finite-time situation}
\label{sec:finite}

For $ T = \infty $, it was possible to rigorously define stable and unstable manifolds, and therefore their local tangent vectors
\edit{were} well-defined.  If $ T < \infty $, these cannot be defined in the normal way.  Is it possible to modify the results for this situation?  Under the nearly autonomous ansatz, this {\em is} possible in a self-consistent way.

Suppose $ T < \infty $, and velocity data $ u(x,t) $ is available for $ \Omega \times [-T,T] $, and there is confidence that
the nearly autonomous ansatz is reasonable.  For the sake of simplicity we shall 
still write $ x \in \Omega $ and
$ t \in [-T,T] $, where for typical applications $ x $ will live on some discrete grid over $ \Omega $, and
$ t $ will be a discrete sampling of points in $ [-T,T] $.  It may be possible to decompose $ u(x,t) $ as a sum of a steady and a small
unsteady velocity quickly because of prior knowledge (and this is what shall be done in Section~\ref{sec:taylorgreen}).  If not,
a decomposition might be performed as follows, and checked for validity.   Define
\[
f(x) := \frac{1}{2T} \int_{-T}^T u(x,t) \, \d t \quad {\mathrm{and}} \quad
g(x,t) := u(x,t) - f(x) \, ; 
\]
in these and in other calculations outlined in this section, it is understood that suitable discrete versions (i.e., a Simpson's
rule evaluation of the integral above) would be necessary.
Letting
\[
\left\| f \right\| := \sup_{\Omega} \left| f(x) \right| \quad {\mathrm{and}} \quad
\left\| g \right\| := \sup_{\Omega \times [-T,T]} \left| g(x,t) \right| \, , 
\]
the data shall be thought of as coming from a nearly autonomous velocity field if the norm of $ g $ is much smaller than the
norm of $ f $.  To be specific, the condition for the nearly autonomous ansatz to be valid is that
$ \eps := \left\| g \right\| \ll \left\| f \right\| $.  If so, there is a base steady flow $ f $ 
which is perturbed \edit{by} the nonautonomous velocity $ g $.  Under the understanding that $ F(x,t,\eps) = f(x) + g(x,t) = u(x,t) $,
this is consistent with the notation of the previous sections, but $ \eps $ here is a derived parameter, and its presence in $ g $
and $ F $ is hidden.  Now, $ g $ is only defined for $ t \in (-T,T) $, and imagine extending to $ \R $ through
\begin{equation}
\renewcommand{\arraystretch}{1.3}
\tilde{g}(x,t) = \left\{ \begin{array}{ll}
g(x,t) & ~{\mathrm{if}} \, t \in [-T,T] \, , \\
0 & ~{\mathrm{if}} \, t \notin [-T,T] \, .
\end{array} \right.
\label{eq:gextend}
\end{equation}
Now consider the {\em infinite-time} flow
\begin{equation}
\dot{x} = f(x) + \tilde{g}(x,t)
\label{eq:infinite}
\end{equation}
for $ (x,t) \in \Omega \times \R $.  What has been done here is that the finite-time nearly autonomous flow has been
extended outside the time domain in which data is available, such that the velocity is steady (with a form derived
from the dominant characteristics of the available velocity field) outside $ (-T,T) $.  This is a reasonable assumption if the
data is obtained from a flow which, for physical or other reasons, is expected to be nearly steady; the `averaged' behaviour
outside of which the data is available can then be assumed to be steady, with a form derived from the data itself.
Given the infinite-time nature of (\ref{eq:infinite}), entities such as hyperbolic
trajectories and stable/unstable manifolds are well-defined for the extended flow (\ref{eq:infinite}), and 
Theorems~\ref{theorem:hyperbolic} and \ref{theorem:eigenvector} apply, subject to the presence of a saddle fixed point
$ a $ of $ f $ such that $ D f(a) $ has eigenvalues $ \lambda_{\tn{s}} < 0 $ and $ \lambda_{\tn{u}} > 0 $.

When the nonautonomy is set to zero in this way, there is a strong connection to the scattering theory development by
Blazevski and collaborators \cite{blazevskidelallave,blazevskifranklin}, who are able to enunciate hyperbolic
trajectories and stable/unstable manifolds in terms of diffeomorphisms of the unperturbed entities, where the diffeomorphism
is expressed in terms of a `scattering map.'  Indeed, they show that for the present perturbative setting, this characterisation
is equivalent to a Melnikov approach \cite[\S3.1]{blazevskidelallave}, which is the basis for the computations of the present
article.  These methods which require the nonautonomy to decay to zero as $ t \rightarrow \pm \infty $ at a sufficiently
fast rate \cite{blazevskidelallave,blazevskifranklin} however do not apply to the infinite-time setting of 
Sections~\ref{sec:hyperbolic}--\ref{sec:eigenvector}, in which $ F(x,t,\eps) $ need not have such decay.

To compute the leading-order hyperbolic trajectory and stable/unstable manifold tangent vector rotations for (\ref{eq:infinite}),
one simply needs to set $ g $ to zero outside $ [-T,T] $ in Theorems~\ref{theorem:hyperbolic} and \ref{theorem:eigenvector}. Theorem~\ref{theorem:hyperbolic} gives the {\em leading-order} (in the nonautonony parameter $ \eps $) expression for the hyperbolic trajectory as
\begin{equation}
a^\star(t,T) := a +  \left[ \alpha_{\tn{u}}^\star(t,T) v_{\tn{u}}^\perp  + 
\frac{ \alpha_{\tn{u}}^\star(t,T) \left( v_{\tn{u}} \cdot v_{\tn{s}} \right) - \alpha_{\tn{s}}^\star(t,T) }{
\ v_{\tn{u}}^\perp \cdot v_{\tn{s}} } \, v_{\tn{u}} 
\right] \, , 
\label{eq:hyperbolicfinite}
\end{equation}
where
\begin{equation}
\alpha_{\tn{s}}^\star(t,T) := - e^{\lambda_{\tn{u}} t} \int_t^T e^{- \lambda_{\tn{u}} \tau} g_{\tn{s}}^{\edit{\dagger}} \left( a, \tau \right)   \, \d \tau
\label{eq:alphasfinite}
\end{equation}
and
\begin{equation}
\alpha_{\tn{u}}^\star(t,T) :=  e^{\lambda_{\tn{s}} t} \int_{-T}^t  e^{ -\lambda_{\tn{s}} \tau} g_{\tn{u}}^{\edit{\dagger}} \left( a, \tau \right)  \, \d \tau \, .
\label{eq:alphaufinite}
\end{equation}
Since the original data was for $ t \in [-T,T] $, and thus the time-derivative would be valid in $ (-T,T) $, the quantity $ a^\star(t,T) $ would give an approximation for the hyperbolic
trajectory in the restricted time-domain $ (-T,T) $.  The superscript $ \star $ will be used hereafter to denote leading-order
approximations for finite-time analogues of entities.

What if $ g $ is extended in a different way to $ \R $?  Then the integral limits in (\ref{eq:alphasfinite}) and (\ref{eq:alphaufinite}) do not get clipped outside $ [-T,T] $, and thus will lead to a different hyperbolic trajectory to leading-order.
If $ g $ is extended to $ \R $ not by setting to zero, but by still following the reasonable hypothesis that $ \| g \| \le \eps $
(which was true in $ [-T,T] $), then the error between using (\ref{eq:alphasfinite}) and the correct $ g $, when projected 
in the $ v_{\tn{s}}^\perp $ direction, is bounded by
\[
\left| A_{\tn{s}}^\star(t,T) \right| \le  \eps e^{\lambda_{\tn{u}} t} \left| \int_T^\infty e^{- \lambda_{\tn{u}} \tau}   \, \d \tau \right| = 
\frac{\eps e^{\lambda_{\tn{u}}(t-T)}}{\lambda_{\tn{u}}} \, , 
\]
and similarly $ \left| A_{\tn{u}}^\star(t,T) \right| \le - \eps e^{\lambda_{\tn{s}}(t-T)}/\lambda_{\tn{s}} $ for the error in the $ v_{\tn{u}}^\perp $ direction.
These furnish bounds, to leading-order in the nonautonomous parameter, for possible extensions.  This approach of
attempting to characterise the effect of velocities from {\em outside} the interval in which data is available is 
an important aspect of finite-time analyses which has not been addressed until this, admittedly fairly limited, analysis.

It should be noted that there are several other suggestions for defining finite-time hyperbolic trajectories in terms
of exponential dichotomies, since 
(\ref{eq:exponentialdichotomy}) is trivially satisfied when $ t $ is restricted to a finite domain, and so `exponentially
decaying in finite-time' would require a stronger condition such as for example insisting on $ K_{\tn{s,u}} = 1 $ 
\cite{doan,karrasch,ducsiegmund,berger}.  This is a strong restriction.  The present approach is an alternative which has applicability if
the nonautonomy is small, and if there is sufficient confidence in the fact that the velocity does not change unduly outside
the interval in which data is available.

By using the extension (\ref{eq:gextend}) and the infinite-time flow (\ref{eq:infinite}), stable and unstable manifolds are
well-defined, and therefore so are their location tangent vectors.  Using Theorem~\ref{theorem:eigenvector}, their directions
 in a time-slice $ t \in (-T,T) $, to leading-order in the nonautonomous parameter,  will be given 
by anticlockwise rotational angles
\begin{equation}
\theta_{\tn{s}}^\star (t,T) := - \, e^{(\lambda_{\tn{u}}-\lambda_{\tn{s}})t} \int_t^T e^{(\lambda_{\tn{s}} - \lambda_{\tn{u}})\tau} v_{\tn{s}}^{\edit{\top}} \edit{D}  g_{\tn{s}}^{\edit{\dagger}} (a,\tau)  \, \d \tau  
\label{eq:slopesfinite}
\end{equation}
and
\begin{equation}
\theta_{\tn{u}}^\star (t,T)  :=   \, e^{(\lambda_{\tn{s}}-\lambda_{\tn{u}})t} \int_{-T}^t e^{(\lambda_{\tn{u}} - \lambda_{\tn{s}})\tau} v_{\tn{u}}^{\edit{\top}} \edit{D}  g_{\tn{u}}^{\edit{\dagger}} (a,\tau)  \, \d \tau  .
\label{eq:slopeufinite}
\end{equation}
of $ v_{\tn{s}} $ and $ v_{\tn{u}} $ respectively.
As $ \theta_{\tn{s}}^\star(T,T) = \theta_{\tn{u}}^\star(-T,T) = 0 $, 
the stable and unstable finite-time tangent vectors evolve \edit{continuously}
in $ (-T,\infty) $ and $ (-\infty, T) $  respectively.  One inevitable factor in this process is that since the data was only
available on $ [-T,T] $, the `guess' used for the data outside $ [-T,T] $ will affect the computed stable and unstable
manifolds.  As shown by Sandstede et~al \cite{finite}, the errors resulting
from extending $ g $ outside $ (-T,T) $ in a nontrivial but bounded way would imply that the invariant manifolds can only be characterised 
as `fat curves;' such nonuniqueness for finite time has also been identified and discussed in alternative ways
\cite{hallerpoje,miller,unsteady,branickiwiggins2010}.  

In what way can the fact that the data is limited to a finite time domain be used
to characterise how the flow entities would change {\em if data were available from outside the interval?}.  After all, in
many problems, practitioners are forced to work with a finite-time data set over some interval, when of course the Lagrangian
flow has been/will be impacted by velocities from outside that interval which are not available.
This is indeed examined in Section~\ref{sec:taylorgreen}
and compared with numerical computations, in a situation when $ \theta_{s}(t) $ is known. 
In general, it may not be.  Consider, then, a situation in which $ g $ is extended outside of $ (-T,T) $ in a nontrivial,
but still `reasonable' way, of letting the velocity shear be bounded in the form  $ \left| \sigma_{\tn{s,u}}(t) \right| \le \eps S $ for all $ t \in \R $.  The errors in using the zero-$ g $ approximations  can then be approximated.
If $ 
E_{\tn{s,u}}^{\star}(t,T):= \theta_{\tn{s,u}}^{\star}(t,T,) - \theta_{\tn{s,u}}(t,\infty) $,  then
\[
E_{\tn{s}}^{\star}(t,T) = e^{(\lambda_{\tn{u}}-\lambda_{\tn{s}})t} \int_T^\infty e^{(\lambda_{\tn{s}}-\lambda_{\tn{u}})\tau} \sigma_{\tn{s}}(\tau) \, \d \tau
\]
and
\[
E_{\tn{u}}^{\star} (t,T) = - e^{(\lambda_{\tn{s}}-\lambda_{\tn{u}})t} \int_{-\infty}^{-T} e^{(\lambda_{\tn{u}}-\lambda_{\tn{s}})\tau} \sigma_{\tn{u}}(\tau) \, \d \tau \, .
\]
from which it is possible to determine the error bounds
\begin{equation}
\left| E_{\tn{s}}^*(t,T) \right| \le \frac{ \eps S e^{(\lambda_{\tn{u}}-\lambda_{\tn{s}})(t-T)}}{\lambda_{\tn{u}} - \lambda_{\tn{s}}} \quad
{\mathrm{and}} \quad
\left| E_{\tn{u}}^*(t,T) \right| \le \frac{ \eps S e^{(\lambda_{\tn{s}}-\lambda_{\tn{u}})(t+T)}}{\lambda_{\tn{u}} - \lambda_{\tn{s}}} \, .
\label{eq:finitebound}
\end{equation}
Generically, exponential decay is to be expected in the finiteness parameter $ T $, with rate given by the {\em difference in the eigenvalues}.  Since to leading-order these are approximated by forward and backwards time finite-time
Lyapunov exponents which can be computed from the data, (\ref{eq:finitebound}) gives a method for estimating the
behaviour of the error due to the finiteness $ T $ of the data.

\subsection{Nonautonomously controlling manifold directions}
\label{sec:control}

This section addresses the inverse question of determining the
velocity required to ensure that the stable and unstable manifolds rotate nonautonomously by specified time-varying
angles.
Consider determining the control velocity $ c(x,t) $ such that
\begin{equation}
\dot{x} = f(x) + c(x,t)
\label{eq:control}
\end{equation}
moves the hyperbolic trajectory from $ a $ to a {\em specified} nearby time-varying location $ \tilde{a}(t) $, and simultaneously
rotates the \edit{local tangents} $ v_{\tn{s,u}} $ associated with $ c(x,t) \equiv 0 $ by {\em specified}, nonautonomously
changing, small angles $ \tilde{\theta}_{\tn{s,u}}(t) $.   

\begin{theorem}[Controlling hyperbolic trajectory location \cite{saddle_control}]
\label{theorem:saddlecontrol}
Let $ T = \infty $, and suppose (\ref{eq:control}) with $ c(x,t) \equiv 0 $ has a saddle fixed  point $ a $ with 
eigensystem $ \left\{ \lambda_{\tn{s,u}}, v_{\tn{s,u}} \right\} $.  Let $ \tilde{a}(t) $ be specified such that  
$ \left| \tilde{a}(t) - a \right| 
+ \left| \tilde{a}'(t) \right| \le \eps $ for all $ t \in \R $.  If $ c(x,t) $ is chosen such that
\begin{equation}
\left. \begin{array}{l}
c(a,t)^{\edit{\top}} v_{\tn{s}}^\perp = \left[ \tilde{a}'(t) - \lambda_{\tn{u}} \left( \tilde{a}(t) - a \right) \right]^{\edit{\top}} v_{\tn{s}}^\perp \\
c(a,t)^{\edit{\top}} v_{\tn{u}}^\perp = \left[ \tilde{a}'(t) - \lambda_{\tn{s}} \left( \tilde{a}(t) - a \right) \right]^{\edit{\top}}  v_{\tn{u}}^\perp
\end{array} \right\} \, , 
\label{eq:saddlecontrol}
\end{equation}
and also subject to the presence of $ A $ such that
\begin{equation}
\left| c(x,t) \right| + \left| D c(x,t) \right| \edit{+ \left| D^2 c(x,t) \right| }+ \left| \frac{\partial c(x,t)}{\partial t} \right| \le \eps A \quad {\mathrm{for}} \, 
(x,t) \in \Omega \times \R \, , 
\label{eq:controlbound}
\end{equation}
then there is a $ K  $ and an actual hyperbolic trajectory $ a(t) $ of (\ref{eq:control}) such that
\[
\left| \tilde{a}(t) - a(t) \right| \le \eps^2 K \quad{\mathrm{for~all}} \, t \in \R \, .
\]
\end{theorem}
\begin{proof}
This result already appears in the literature in a slightly different form \cite{saddle_control}; formal methods for achieving higher-order accuracy \cite{controlnd}, or of stabilising the trajectory \cite{yagasakicontrol} are also available.
\end{proof}

Theorem~\ref{theorem:saddlecontrol} relates to `inverting' Theorem~\ref{theorem:hyperbolic}.  Of note here is the
fact that there is no $ \eps $ explicit to the velocity field, but rather $ \eps $ is a parameter representing how large the deviation
of the required hyperbolic trajectory $ \tilde{a} $ is from the uncontrolled hyperbolic fixed point $ a $.  Under the specified form
(\ref{eq:saddlecontrol}) of the control velocity, the \edit{ $ {\mathrm{C}}^0 $-norm error of using the desired hyperbolic trajectory as the actual one} will be of order $ \eps^2 $ \edit{in the following sense}.  \edit{If a particular hyperbolic trajectory with $ \eps = 0.1 $ is specified, Theorem~\ref{theorem:saddlecontrol} indicates that the error between the required and actual hyperbolic trajectory 
resulting from applying the control velocity in (\ref{eq:saddlecontrol}) would be bounded by $ 0.1 K $, where $ K $ is
independent of $ t $.  On the other hand, if {\em another} hyperbolic trajectory with $ \eps = 0.01 $ is specified and
$ c $ is chosen subject to (\ref{eq:saddlecontrol}), then the error between the actual and desired trajectory would be
bounded by $ 0.0001 K $, for {\em exactly the same $ K $}.} 
Next, the main contribution of this article towards a control strategy---manipulating the directions at which the stable and unstable
manifold emanate---is stated.

\begin{theorem}[Controlling local manifold directions]
\label{theorem:control}
Under the hypotheses of Theorem~\ref{theorem:saddlecontrol}, 
suppose also $ \tilde{\theta}_{s,u}: \R \rightarrow \R $ such that
$ \left| \tilde{\theta}_{\tn{s,u}}(t) \right| + \left| \tilde{\theta}_{\tn{s,u}}'(t) \right| < \eps $ for all $ t $.  
If $ c(x,t) $ is chosen subject to the velocity shear conditions
\begin{equation}
\left. v_{\tn{s}}^{\edit{\top}}  \edit{D} \left[ c(x,t)^{\edit{\top}}  v_{\tn{s}}^\perp \right] \right|_{x=a}  = 
 \tilde{\theta}_{\tn{s}}'(t)  - \left( \lambda_{\tn{u}} - \lambda_{\tn{s}} \right) \tilde{\theta}_{\tn{s}}(t)
 \label{eq:controls}
\end{equation}
and
\begin{equation}
 \left. v_{\tn{u}}^{\edit{\top}} \edit{D} \left[ c(x,t)^{\edit{\top}}  v_{\tn{u}}^\perp \right] \right|_{x=a}  = 
 \tilde{\theta}_{\tn{u}}'(t)  + \left( \lambda_{\tn{u}} - \lambda_{\tn{s}} \right) \tilde{\theta}_{\tn{u}}(t) \, ,
\label{eq:controlu}
\end{equation}
and \edit{also subject to the smoothness condition (\ref{eq:controlbound}),} then there exist $ K_{\tn{s,u}} $ such that the actual rotational angles $ \theta_{\tn{s,u}}(t) $ of the stable and unstable manifolds at
the hyperbolic trajectory location $ a(t) $  of (\ref{eq:control}) satisfy
\[
\left| \theta_{\tn{s}}(t) - \tilde{\theta}_{\tn{s}}(t) \right| \le  \eps^2 K_{\tn{s}} \quad {\mathrm{and}} \quad
\left| \theta_{\tn{u}}(t) - \tilde{\theta}_{\tn{u}}(t) \right| \le \eps^2 K_{\tn{u}} \quad {\mathrm{for}} \, \edit{t \in \R} \, .  
\]
\end{theorem}

\begin{proof}
Note that (\ref{eq:control}) is the same as (\ref{eq:flow}) under the identification $ F(x,t,\eps) = f(x) + c(x,t) $, and
thus $ g(x,t,\eps) = c(x,t) $.
However, there is no explicit $ \eps $-dependence in the vector field, which is only proper given that no such 
$ \eps $-dependence was specified in the required rotations, save for the fact that these rotations are of {\em maximum} size
$ \eps $.  
The result of Theorem~\ref{theorem:eigenvector} applied to this $ c $ leads to
\[
\theta_{\tn{s}}(t) = 
- \, e^{(\lambda_{\tn{u}}-\lambda_{\tn{s}})t} \int_t^\infty e^{(\lambda_{\tn{s}} - \lambda_{\tn{u}})\tau} \sigma_{\tn{s}}(\tau)  \, \d \tau  + 
\eps^2 E_{\tn{s}}(t,\eps) \, ,
\]
\edit{where $ E_{\tn{s}} $ and its $ t $-derivative are uniformly bounded for $ t \in \R $}, for the {\em actual} rotation of the local stable manifold.  Here, the 
velocity shear is, from (\ref{eq:shear}),
\[
\sigma_{\tn{s}}(t) = \edit{  v_{\tn{s}}^\top D c(a,t) v_{\tn{s}}^\perp } = 
\left. v_{\tn{s}} \cdot \nabla \left[ c(x,t) \cdot v_{\tn{s}}^\perp \right] \right|_{x=a} \, .
\]
  Therefore, 
\[
e^{(\lambda_{\tn{s}} - \lambda_{\tn{u}})t} \theta_{\tn{s}}(t) = 
- \, \int_t^\infty e^{(\lambda_{\tn{s}} - \lambda_{\tn{u}})\tau} \sigma_{\tn{s}}(\tau)  \, \d \tau  + 
\eps^2 E_{\tn{s}}(t,\eps) e^{(\lambda_{\tn{s}} - \lambda_{\tn{u}})t} \, ,
\]
which upon differentiating with respect to $ t $ gives
\[
e^{(\lambda_{\tn{s}} - \lambda_{\tn{u}})t} \left[ \theta_{\tn{s}}'(t) + \left( \lambda_{\tn{s}} - \lambda_{\tn{u}} \right) \theta_{\tn{s}}(t) \right] = 
 e^{(\lambda_{\tn{s}} - \lambda_{\tn{u}})t} \sigma_{\tn{s}}(t)    + 
\eps^2  e^{(\lambda_{\tn{s}} - \lambda_{\tn{u}})t} \left[ \frac{\partial E_{\tn{s}}(t,\eps)}{\partial t} + \left( \lambda_{\tn{s}} - \lambda_{\tn{u}} \right) 
E_{\tn{s}}(t,\eps) \right] \, .
\]
The actual velocity shear is therefore
\[
\sigma_{\tn{s}}(t) = \theta_{\tn{s}}'(t) + \left( \lambda_{\tn{s}} - \lambda_{\tn{u}} \right) \theta_{\tn{s}}(t) - \eps^2 \left[ \frac{\partial E_{\tn{s}}(t,\eps)}{\partial t} + 
\left( \lambda_{\tn{s}} - \lambda_{\tn{u}} \right) E_{\tn{s}}(t,\eps) \right] \, , 
\]
in terms of the {\em actual} rotation angle $ \theta_{\tn{s}} $.  On the other hand, the control velocity was chosen subject to
(\ref{eq:controls}) and hence
\[
\tilde{\theta}_{\tn{s}}'(t) + \left( \lambda_{\tn{s}} - \lambda_{\tn{u}} \right) \tilde{\theta}_{\tn{s}}(t)  = \theta_{\tn{s}}'(t) + \left( \lambda_{\tn{s}} - \lambda_{\tn{u}} \right) \theta_{\tn{s}}(t) - \eps^2 \left[ \frac{\partial E_{\tn{s}}(t,\eps)}{\partial t} + 
\left( \lambda_{\tn{s}} - \lambda_{\tn{u}} \right) E_{\tn{s}}(t,\eps) \right] \,  .
\]
Upon defining $ \eta(t) := \tilde{\theta}_{\tn{s}}(t) - \theta_{\tn{s}}(t) $, the differential equation
\[
\eta'(t) + \left( \lambda_{\tn{s}} - \lambda_{\tn{u}} \right) \eta(t)  = - \eps^2 \left[ \frac{\partial E_{\tn{s}}(t,\eps)}{\partial t} + 
\left( \lambda_{\tn{s}} - \lambda_{\tn{u}} \right) E_{\tn{s}}(t,\eps) \right] =: \eps^2 \tilde{E}_s(t,\eps) \, 
\]
results, where by Theorem~\ref{theorem:eigenvector}, $ \tilde{E}_s(t,\eps) $ is bounded  \edit{ for $ t \in \R $}.  Multiplying
by the integrating factor and integrating from a general time $ t $ to $ \infty $ yields
\[
\lim_{\tau \rightarrow \infty} \left[ e^{(\lambda_{\tn{s}} - \lambda_{\tn{u}})\tau} \eta(\tau)  \right] - e^{(\lambda_{\tn{s}} - \lambda_{\tn{u}})t} \eta(t) = \eps^2
\int_t^\infty e^{(\lambda_{\tn{s}} - \lambda_{\tn{u}})\tau} \tilde{E}_s(\tau,\eps) \, \d \tau \, .
\]
The limit on the left is zero because $ \theta $ and $ \tilde{\theta} $, and consequently $ \eta $, are bounded on \edit{$ \R  $},
and $ \lambda_{\tn{s}} - \lambda_{\tn{u}} < 0 $.
By virtue of the bound (call it $ \tilde{K}_{\tn{s}} $) of $ \tilde{E} $, 
\[
e^{(\lambda_{\tn{s}} - \lambda_{\tn{u}})t} \left| \eta(t) \right| \le \eps^2 \tilde{K}_{\tn{s}}  \int_t^\infty e^{(\lambda_{\tn{s}} - \lambda_{\tn{u}})\tau} \, \d \tau 
= \eps^2 \frac{\tilde{K}_{\tn{s}} e^{(\lambda_{\tn{s}} - \lambda_{\tn{u}})t} }{\lambda_{\tn{u}} - \lambda_{\tn{s}}} \, , 
\]
and therefore $ \left| \eta(t) \right| \le \eps^2 \tilde{K}_{\tn{s}} / (\lambda_{\tn{u}} - \lambda_{\tn{s}}) =: \eps^2 K_{\tn{s}} $, proving that the actual and the desired
rotations are bounded by an $ {\mathcal O}(\eps^2) $ quantity for $ t \in \edit{\R} $.  
The development so far was only for the rotation of the stable manifold; the expression (\ref{eq:controlu})
is similarly derived from (\ref{eq:slopeu}).
\end{proof}

Theorem~\ref{theorem:control} shows that \edit{local manifolds} can be
controlled via imposing a prescribed velocity shear \edit{at $ a $}.   The proof strategy above is similar to that previously used to control the hyperbolic trajectory locations \cite{saddle_control};
this has now been extended to be able to also control the directions of emanation of invariant manifolds.

\begin{remark}[Finite-time control of manifold directions] \normalfont
\label{remark:control}
Let $ T < \infty $, and suppose $ \tilde{\theta}_{s,u}^\star(t) $ is only specified for $ t \in (-T,T) $.  Then, by choosing
$ \tilde{\sigma}_{\tn{s,u}}^\star(t) $ for $ t \in (-T,T) $ as in Theorem~\ref{theorem:control}, the expectation is that corresponding
finite-time stable and unstable manifolds will emanate in the directions specified by $ \theta_{\tn{s,u}}^\star(t) $ to
leading-order in the nonautonomy parameter.  However, the
finiteness {\em will} contribute to a $ {\mathcal O}(\eps) $ error; see (\ref{eq:finitebound}), for example.

\end{remark}

\section{Implementation and verification}
\label{sec:taylorgreen}

\begin{figure}[t]
\begin{center}
\includegraphics[width=0.7 \textwidth, height=0.3 \textheight]{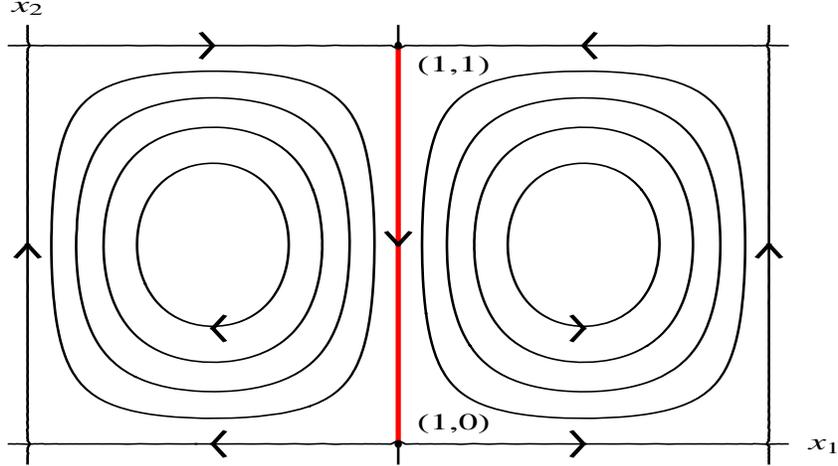}
\caption{The Taylor-Green flow
(\ref{eq:taylorgreen}); the key entity [heavy line] is the heteroclinic connection from $ (1,1) $ to $ (1,0) $. }
\label{fig:taylorgreen}
\end{center}
\end{figure}

This section numerically investigates the theory of the previous sections, specifically focussing on a finite-time
setting.  Firstly, a control condition on how the manifolds emanate is considered, and then numerically evaluated using 
finite-time Lyapunov exponents.  
Secondly, the influence of finiteness of time is assessed.
The system chosen for investigation is the well-known Taylor-Green flow \cite{chandrasekhar,taylorgreen,periodic}
\begin{equation}
\renewcommand{\arraystretch}{1.2}
\left.
\begin{array}{lcl}
\dot{x}_1 & = & - \pi A \, \sin \left( \pi x_1 \right) \, \cos \left( \pi x_2 \right) \\
\dot{x}_2 & = & \pi A \, \cos \left( \pi x_1 \right) \, \sin \left( \pi x_2 \right) \,
\end{array}
\right\}
\label{eq:taylorgreen}
\end{equation}
whose phase portrait (Fig.~\ref{fig:taylorgreen}) 
has a stable manifold coming vertically downwards to the saddle fixed point
$ (1,0) $, which is simultaneously a branch of the unstable manifold emanating downwards from the point
$ (1,1) $.  The relevant eigenvalues are $ \pm \pi^2 A $ at both these points.  The break up of these stable and unstable
manifolds under a perturbation results in transport between the left gyre $ [0,1] \times [0,1] $ and the right gyre
$ [1,2] \times [0,1] $.  

Being able to simultaneously control the two splitting manifolds has an impact on the transport between the two gyres.
Suppose a finite-time nonautonomous perturbation is to be introduced to (\ref{eq:taylorgreen}) such that the local
stable and unstable manifolds rotate by angles $ \tilde{\theta}_{\tn{s}}(t) $ and $ \tilde{\theta}_{\tn{u}}(t) $ respectively,
and the hyperbolic trajectories perturbing from $ (1,0) $ and $ (1,1) $ also follow a specified time-variation, where the
variations of these from the unperturbed situation is bounded by $ \eps $.
If $ c = ( \begin{array}{cc} c_1 & c_2 \end{array} )^\top $ is the control perturbation which achieves this subject to
an error which is bounded according to Theorems~\ref{theorem:saddlecontrol} and \ref{theorem:control}, the system is now
\begin{equation}
\renewcommand{\arraystretch}{1.2}
\left.
\begin{array}{lcl}
\dot{x}_1 & = & - \pi A \, \sin \left( \pi x_1 \right) \, \cos \left( \pi x_2 \right) + c_1 (x_1, x_2, t) \\
\dot{x}_2 & = & \pi A \, \cos \left( \pi x_1 \right) \, \sin \left( \pi x_2 \right) + c_2 (x_1,x_2, t) \,
\end{array}
\right\} \, .
\label{eq:taylorgreencontrol}
\end{equation}
Applying Theorem~\ref{theorem:control} for the stable manifold of the hyperbolic trajectory
near $ (1,0) $ leads to
\[
\frac{\partial c_1}{\partial x_2}(1,0,t) \approx - \left[ \tilde{\theta}_{\tn{s}}'(t) - 2 \pi^2 A \, \tilde{\theta}_{\tn{s}}(t) \right] \, ,
\]
and a similar analysis for the unstable manifold of the hyperbolic trajectory near $ (1,1) $ yields
\[
\frac{\partial c_1}{\partial x_2}(1,1,t) \approx - \left[ \tilde{\theta}_{\tn{u}}'(t) + 2 \pi^2 A \, \tilde{\theta}_{\tn{u}}(t) \right] \, .
\]

\subsection{Time-periodic example}

First, suppose that the hyperbolic trajectories are required to remain at their autonomous locations, but that
the stable and unstable manifold are to be moved in a time-periodic fashion.
From Theorem~\ref{theorem:saddlecontrol}, it is clear that choosing $
c_{1,2}(1,0,t) = 0 $ and $ c_{1,2}(1,1,t) = 0 $
makes the leading-order hyperbolic trajectory movement zero.
While complying with this, the required rotations of the tangent vectors can be realised by choosing 
$ c_2(x_1,x_2,t) \equiv 0 $ and 
\begin{equation}
c_1(x_1,x_2,t) = - \left[\tilde{\theta}_{\tn{s}}'(t) - 2 \pi^2 A \, \tilde{\theta}_{\tn{s}}(t)\right] x_2 I_{(1,0)}(x_1,x_2) 
- \left[ \tilde{\theta}_{\tn{u}}'(t) + 2 \pi^2 A \, \tilde{\theta}_{\tn{u}}(t) \right] (x_2 -1) I_{(1,1)}(x_1,x_2)
\label{eq:tgcontrol}
\end{equation}
where 
\[
I_{(x_1^0,x_2^0)}(x_1,x_2) = \frac{1}{2} \left[ \tanh \frac{\sqrt{(x_1-x_1^0)^2+(x_2-x_2^0)^2} + \delta}{\delta^2}  
- \tanh \frac{\sqrt{(x_1-x_1^0)^2+(x_2-x_2^0)^2} - \delta}{\delta^2} \right]
\]
is used as a \edit{smooth approximation} of an indicator function in a $ \delta $-radius ball around $ (x_1^0,x_2^0) $.
Now, if $ \theta_{\tn{s}} $ and $ \theta_{\tn{u}} $ are both positive, the stable and unstable manifolds will respectively emanate
towards the left gyre and the right gyre.  Any subsequent intersection pattern between them will then span a larger
area than if the manifolds emanated into the {\em same} gyre.  Since this intuition translates to whenever $ \theta_{\tn{s}} $
and $ \theta_{\tn{u}} $ have the same sign, one method for attempting to achieve greater transport would be to have
$ \tilde{\theta}_{\tn{s,u}}(t) $ oscillating in phase. 
To achieve this, choose $ \tilde{\theta}_{\tn{s,u}}(t) = \delta \cos \left( \omega t \right) $, and thus one can
take the bound in Theorem~\ref{theorem:control} to be $ \eps = \delta \left( 1 + \left| \omega \right| \right) $.  This gives
the control velocity
\begin{equation}
\left( \begin{array}{c}
c_1(x_1,x_2,t) \\ \mbox{} \\ c_2(x_1,x_2,t) \end{array} \right) 
=  \delta \left( \begin{array}{c}
\hspace*{-1cm} \left[ \omega \sin \left( \omega t \right) + 2 \pi^2 A \cos \left( \omega t \right) \right] x_2 I_{(1,0)}(x_1,x_2) \\
\hspace*{1cm} + 
\left[ \omega \sin \left( \omega t \right) - 2 \pi^2 A \cos \left( \omega t \right) \right] 
 (x_2-1) I_{(1,1)}(x_1,x_2) \\ 0 \end{array} \right) \, .
\label{eq:tgperiodic}
\end{equation}

\begin{figure}[t]
\includegraphics[width=\textwidth, height=0.3 \textheight]{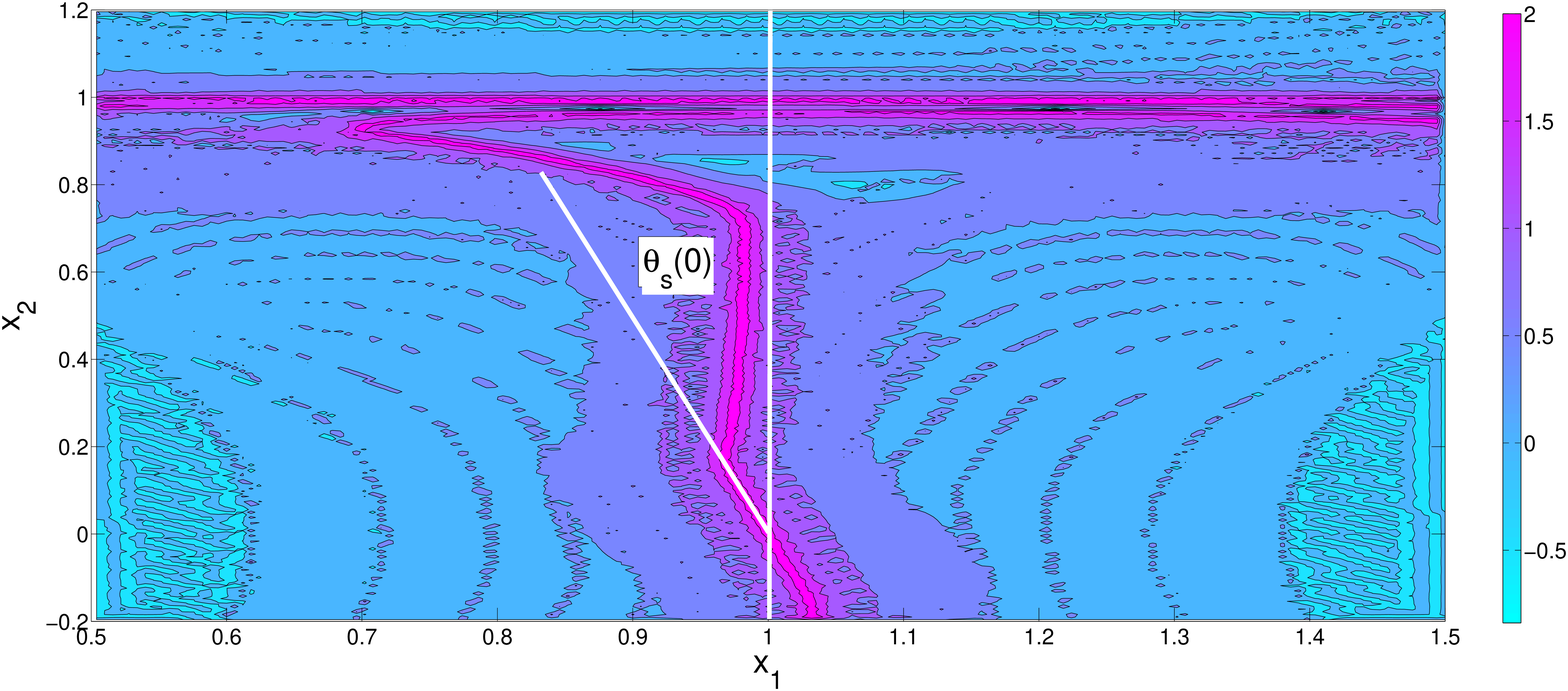}
\includegraphics[width=\textwidth, height=0.3 \textheight]{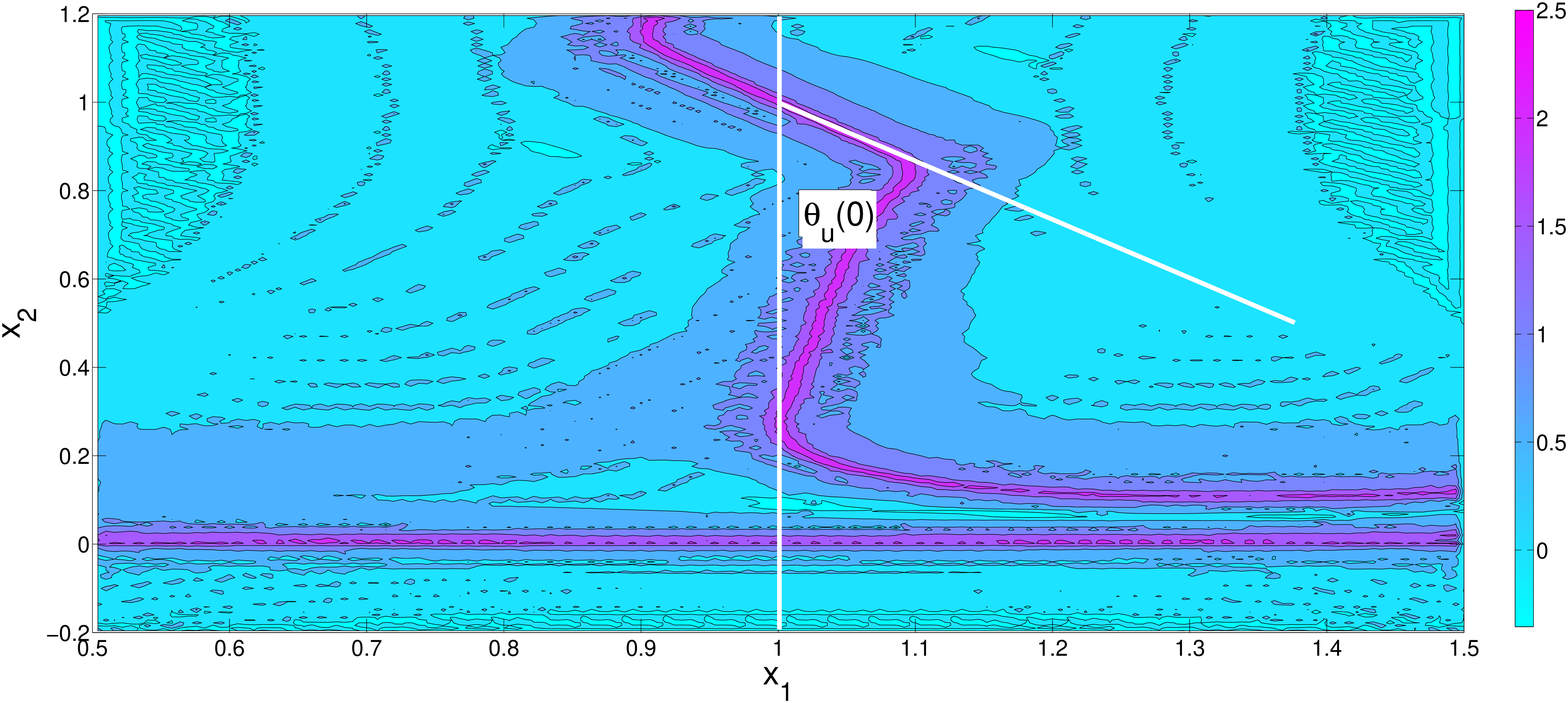}
\caption{Forward-time (top) and backward-time (bottom) FTLEs at $ t = 0 $ for the finite-time control velocity
(\ref{eq:tgperiodic}), with $ A=1 $, $ \omega = 2 \pi $, $ \delta = 0.2 $ and $ T = 2 $. }
\label{fig:ftle}
\end{figure}

To evaluate this method {\em in finite time}, the choice of parameters $ A = 1 $ and $ \omega = 2 \pi $ is used, and
thus $ \eps = \delta (1 + 4 \pi ) $.
The finiteness parameter is chosen as $ T = 2 $, which is double the period of $ c_1 $, and therefore not very large.
The value $ \delta = 0.2 $ is chosen to deliberately examine the efficacy of the method at relatively large $ \eps $
(which is $ 0.2 (1 + 4 \pi) \approx 2.71 $ in this case, in comparison to the unperturbed velocity scale of 
$ \pi A \approx 3.14 $) at
which the perturbative nature of the theory may be compromised.
 A temporal discretisation using $ 101 $ time-slices was used for the time-interval $ [-2,2] $ with a time-spacing
of $ 0.04 $.  FTLE fields were then calculated within each time-slice numerically\footnote{FTLE
computations were performed
using LCS Matlab kit Version~2.3 developed at the Biological Propulsion Laboratory, CalTech, and available at:
\edit{http://dabirilab.com/software/}.}, but using the {\em full} data available at each instance.  For example,
when considering the time-slice $ t = 0.44 $, forward FTLEs included data from $ t = 0.44 $ to $ t = 2.0 $ (a time-interval
of length $ 1.56 $) while backward FTLE used data from $ t = -2.0 $ to $ 0.44 $ (a time-length of $ 2.44 $).  This is
in keeping with the understanding that, given a certain finite-time data set, one would like to use the maximum
information available in that set in performing numerical calculations.  FTLEs were chosen as the 
diagnostic for finite-time versions of stable and unstable manifolds since when focussing near the relevant point,
their ridges appeared unambigiously.  Thus, technical modifications to counteract possible misdiagnoses by FTLEs
\cite{branickiwiggins2009,karraschhaller,shadden,norgardbremer,schindler,hallervariational} were not needed.  Moreover, 
FTLEs in their simplest sense are possibly the most commonly used diagnostic method in the literature, even though
other methods, or suitable refinements of existing methods 
\cite[e.g.]{hallerreview,karraschattract,nelsonjacobs,allshousepeacock,budisicthiffeault,fortin}, continue to be developed.
The results for $ t = 0 $ are shown
in Fig.~\ref{fig:ftle}, which shows the angular rotations $ \theta_{\tn{s,u}}(0) $ incurred by the clearly defined FTLE ridges.
Thus by focussing energy in $ \delta $-balls around $ (1,0) $ and $ (1,1) $
in a judiciously selected fashion, global transport has therefore been enhanced.

\begin{figure}[t]
\includegraphics[width=0.7 \textwidth, height=0.3 \textheight]{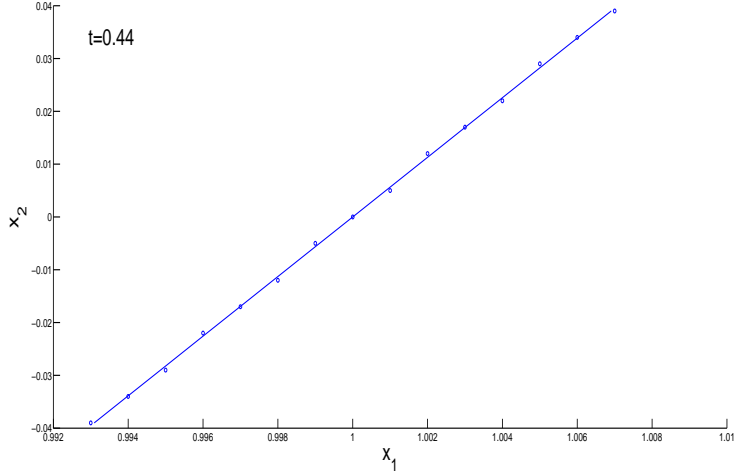}
\caption{Points on the extracted FTLE ridge [circles] near $ (1,0) $ in the time-slice $ t = 0.44 $; the corresponding
linear fit is shown by the straight line.}
\label{fig:ridgeextract}
\end{figure}

Henceforth, attention will be focussed on evaluating the accuracy of the stable manifold.  At each $ t $ value, 
points on the Lyapunov exponent ridge were extracted by zeroing in to the region 
$ -0.04 < x_2 < 0.04 $ in the vicinity of $ (1,0) $, and picking points from the FTLE field which lie above a cut-off threshold
($0.95$
of the maximum FTLE value).  This simple-minded ridge-extraction algorithm is sufficient for the purposes of computations
in this article, since a dominant ridge is present in the vicinity of $ (1,0) $; if not, more sophisticated approaches would be
necessary \cite{karraschhaller,schindler}.
An example is shown for $ t = 0.44 $ in Fig.~\ref{fig:ridgeextract}; the points essentially lie
along a straight line which, in this case, corresponds to a negative $ \theta_{\tn{s}} $ since the rotation is clockwise from the
vertical.  The slope
and intercept of this line were calculated using standard linear regression.  The negative reciprocal of the slope gives
$ \tan \theta_{\tn{s}} $, while the intercept can be used to compute the location of the perturbed hyperbolic trajectory.  
The $ x_1 $ coordinate of the hyperbolic trajectory variation with $ t $ is shown in the left panel of Fig.~\ref{fig:periodic}.  This is preserved
near $ x_1 = 1 $ with very high accuracy, as expected with the choice of $ c_1 = c_2 = 0 $ at $ (1,0) $ and Theorem~\ref{theorem:hyperbolic}.
The values of the computed $ \theta_{\tn{s}} $s from the FTLE ridge extraction procedure
are displayed in the right panel of Fig.~\ref{fig:periodic} as circles, for $ t $ values in $ [0,1] $.  
The numerical calculations were performed independently at each $ t $ value, to not prejudice their comparison to
the desired tangent vectors at independent times, rather than using improvements to FTLE ideas \cite{karraschattract} 
in which Lagrangian advection can be used to advantage.
These computed values of $ \theta_{\tn{s}}(t) $ 
are remarkably close to
the specified curve $ \tilde{\theta}_{\tn{s}}(t) = \delta \cos \left( \omega t \right) $, illustrating that the control strategy used is
highly effective even at this relatively large value of $ \eps $.  This offers evidence that
(\ref{eq:slopesfinite}) and (\ref{eq:slopeufinite}) offer excellent approximations in this case.

\begin{figure}[t]
\includegraphics[width=0.4 \textwidth, height=0.3 \textheight]{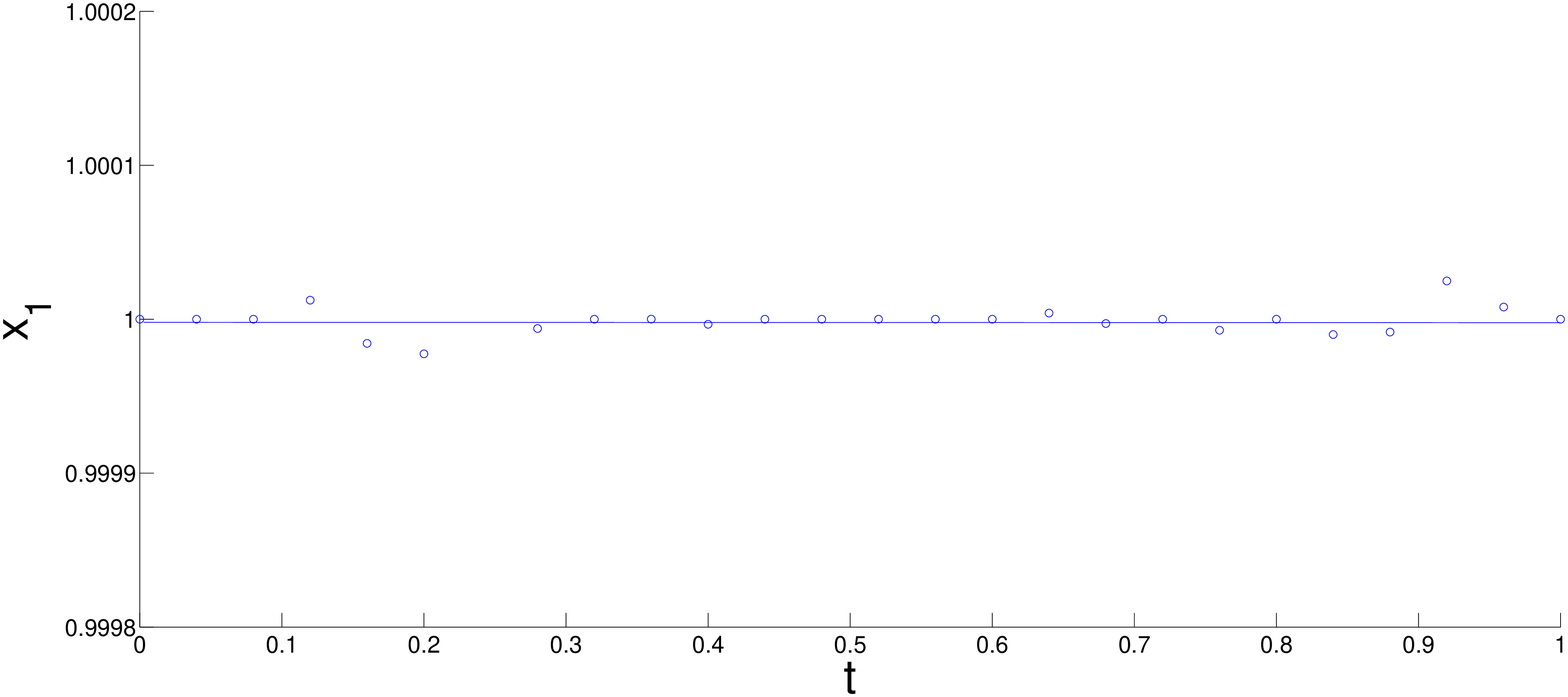}
\includegraphics[width=0.55 \textwidth, height=0.3 \textheight]{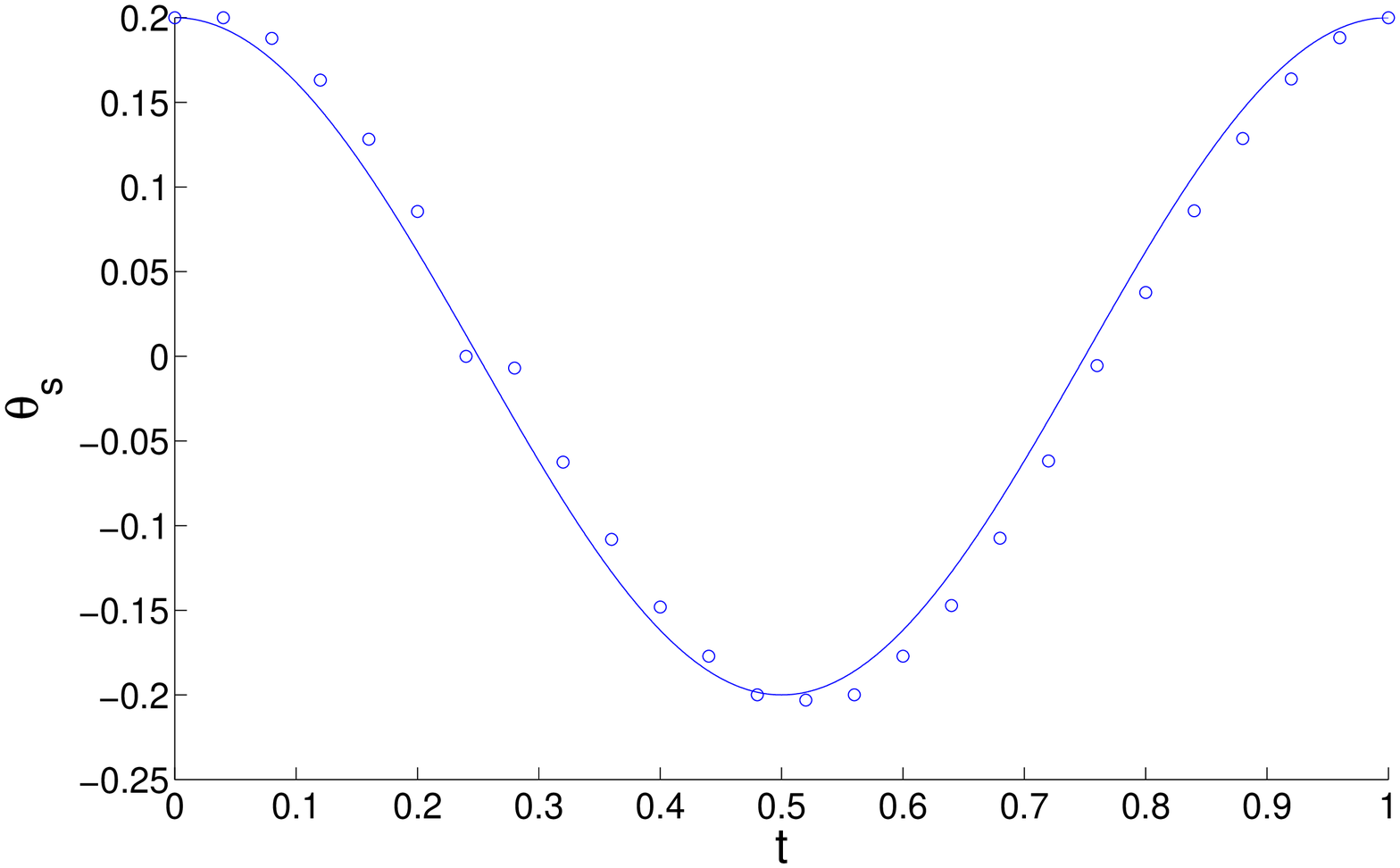}
\caption{Results for the control velocity (\ref{eq:tgperiodic}) with $ \delta = 0.2 $, $ A = 1 $,
$ \omega = 2 \pi $ and $ T = 2 $, extracted by computing the intercept and slope of the FTLE ridges using linear regression [circles]:
$ x_1 $-coordinate of hyperbolic trajectory (left panel) and $ \theta_{\tn{s}} $ (right panel) in comparison
with the expected curves [solid].}
\label{fig:periodic}
\end{figure}

The theoretical results indicated that the error would be of order $ \eps^2 $, which since $ \eps = \delta(1 + 4 \pi) $
is equivalent to the statement that it is of order $ \delta^2 $ in this situation.  To test this, the quantity $ E = 
\left| \tilde{\theta}_{\tn{s}}(0) - \theta_{\tn{s}}(0) \right| $ where $ \theta_{\tn{s}}(0) $ is that computed via the FTLE ridges and their slopes, was determined for
different values of $ \delta $.  The results are shown in the log-log plot of Fig.~\ref{fig:epsilon}.  
While the log-log plot
only approximately a straight line, it indicates that the error is approximately $ {\mathcal O}(\eps^{2.5}) $, which is consistent
with Theorem~\ref{theorem:eigenvector}.  It was observed that the 
spatial grid resolution is unable to see the perturbation if $ \delta \lesssim 0.2 $, below which the numerical computations return
$ \theta_{\tn{s}}(0) = 0 $.   A more refined grid will be necessary to push
the calculations to smaller $ \delta $s, resulting in significant computational cost.

\begin{figure}[t]
\centering
\includegraphics[width=0.6 \textwidth, height=0.3 \textheight]{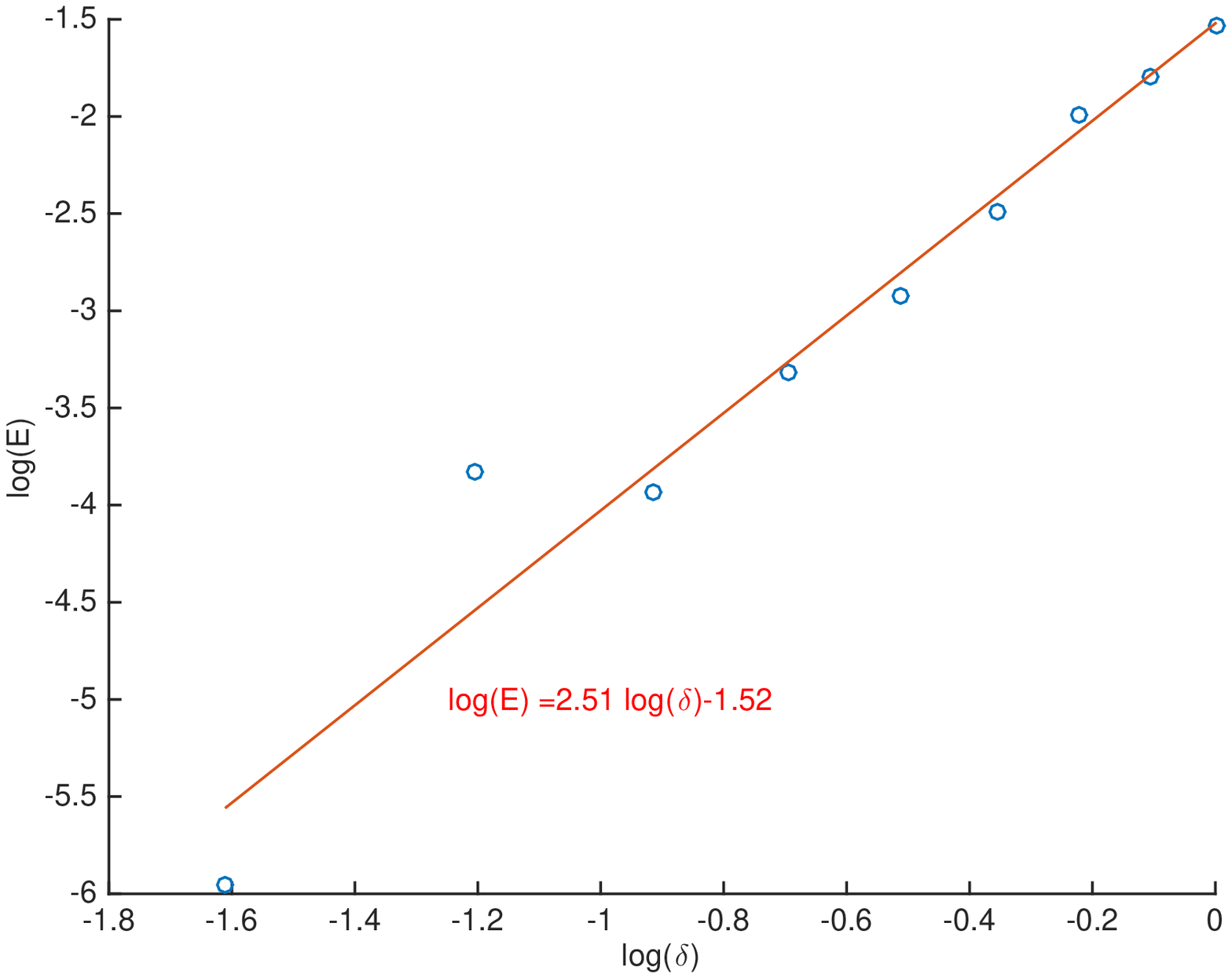}
\caption{Asymptotic order error analysis for the control velocity (\ref{eq:tgperiodic}) with $ A = 1 $,
$ \omega = 2 \pi $ and $ T = 1 $, where $ E = \left| \tilde{\theta}_{\tn{s}}(0) - \theta_{\tn{s}}(0) \right| $, displaying that
$ E \sim {\mathcal O}(\delta^{2.5}) = {\mathcal O}(\eps^{2.5}) $.}
\label{fig:epsilon}
\end{figure}

Next, the evaluation of the error as a consequence of clipping the data at a finite $ T $  was performed.
In keeping with finite-time reality, all the FTLE computations were run up to time $ T = 2 $, the largest value of time for which
data was considered available.  Thus each data point in Fig.~\ref{fig:periodic} was computed
by flowing forward over a {\em different} amount of time ($ T - t $ for differing values of $ t $).   
Assuming that the vector field is only defined
up to time $ T $, simplifying (\ref{eq:slopesfinite}) with these parameters and the desired $ \tilde{\theta}_{\tn{s}} $
gives the leading-order {\em finite-time} rotation of the stable manifold to be
\[
\tilde{\theta}_{\tn{s}}^{\star}(t,T) = \delta \cos \left( \omega t \right) - \delta e^{-2 \pi^2 A(T-t)} \cos \left( \omega T \right)  \, ,
\]
whereas the infinite-time leading-order value is $ \theta(t) = \delta \cos \omega t $.
Therefore, the error related to the finiteness of $ T $ is
\begin{equation}
E_{\tn{s}}^{\star}(t,T) := \tilde{\theta}_{\tn{s}}^{\star}(t,T) - \theta_{\tn{s}}(t) = - \delta e^{-2 \pi^2 A(T-t)} \cos \left( \omega T \right) \, ,
\label{eq:errorT}
\end{equation}
which gives a  estimate---based on the {\em desired} rotation $ \tilde{\theta}_{\tn{s}} $---of how the stable manifold rotation to leading-order in the nonautonomy is impacted by clipping the data at
different $ T $ values.  If fully infinite-time data ($ T = \infty $) is available, the leading-order error goes to zero.  If data 
from a clipped time $ [-T,T] $ is used, the leading-order in $ \eps $ (equivalently, in $ \delta $) stable manifold rotation
incurs an error characterised by (\ref{eq:errorT}).
Let $ t = 0 $ be fixed.  For $ T $ values in the range $ 0.2 $ to $ 1.0 $ in steps of $ 0.04 $, the forward
time FTLE can be computed using data in the interval $ [0,T] $.  For each such $ T $, the actual
$ \theta_{\tn{s}}^*(0,T) $
was computed numerically using the ridge extraction procedure and linear regression, as described before.
The important thing to note is that during each computation, the data is only assumed to be known in the
interval $ [0,T] $, which is different for each $ T $.  Since $ \theta_{\tn{s}}(0) = \delta \cos \left( \omega 0 \right) = \delta $,
the estimate for $ \theta_{\tn{s}}^*(0,T) - \theta_{\tn{s}}(0) $ can be found from the data, which is pictured by the filled
circles in Fig.~\ref{fig:error}.  The curve in Fig.~\ref{fig:error} is the right-hand side of (\ref{eq:errorT}) with $ t = 0 $; 
the differences between the curve and the circles is because the curve uses the {\em desired} value $ \tilde{\theta}_{\tn{s}} $
whereas the circles are the {\em obtained} values of $ \theta_{\tn{s}} $.

\begin{figure}[t]
\includegraphics[width=\textwidth, height=0.3 \textheight]{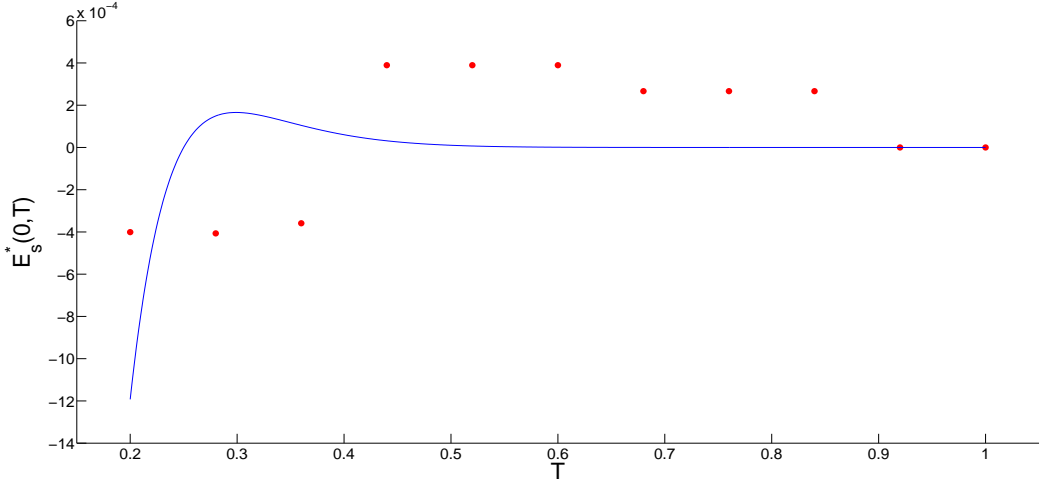}
\caption{Numerically computed values of $ \theta_{\tn{s}}^*(0,T) - \theta_{\tn{s}}(0) $ [filled circles], compared with the
error estimate (\ref{eq:errorT}) [solid curve], with $ \delta = 0.2 $, $ A = 1 $ and $ \omega = 2
\pi $.}
\label{fig:error}
\end{figure}

\subsection{Time-discontinuous example}

Time-periodicity is not a requirement of the theory.  In order to push this advantage further than one could
legitimately hope, suppose
\begin{equation}
\tilde{a}(t) = \left( 1 + \delta(4t\!-\!2) \left[ H(t\!-\!1/4) - H(t\!-\!3/4) \right], 0 \right)^\top \quad {\mathrm{and}} \quad
\tilde{\theta}_{\tn{s}}(t) = \delta \left[ 1 - H(t\!-\!1/2) \right] 
\label{eq:abrupt}
\end{equation}
where $ H(\centerdot) $ is the Heaviside function, and $ \delta > 0 $ a parameter.  This ambitious requirement does {\em not} comply with theoretical conditions
needed;   the bound $ \eps $ arising from $ \tilde{a} $, $ \tilde{\theta}_{\tn{s}} $ and their derivatives in Theorems~\ref{theorem:saddlecontrol} and \ref{theorem:control} does not exist at some points.  Moreover, abrupt switching of locations cannot be
achieved in smooth differential equations, and indeed the definition of a stable manifold collapses.
How well can the theory be used to help to achieve these computationally?

Choose \edit{smooth approximations} of (\ref{eq:abrupt}) given by
\begin{equation}
\tilde{a}(t) \! = \! \left( \! 1 \! + \! \delta(2t\!-\!1) \! \left[ \tanh \! \left( \frac{t \! \!-\! \! 1/4}{\delta^2} \right) \! - \! \tanh \! \left( \frac{t \! \!-\! \! 3/4}{\delta^2} \right) \right], 0 \right)^\top \! 
{\mathrm{and}} \, 
\tilde{\theta}_{\tn{s}}(t) \! = \! \delta \tanh \! \left( \frac{t\! \!-\! \!1/2}{\delta^2} \right) .
\label{eq:abruptsmooth}
\end{equation}
Theorem~\ref{theorem:saddlecontrol} gives the requirement $ c_2(a,t) = 0 $ and 
\begin{eqnarray*}
c_1(a,t) & = & \delta \left( 2 - 2 \pi^2 A t + \pi^2 A \right) \left[ \tanh \left( \frac{t-1/4}{\delta^2} \right) - \tanh \left( \frac{t-3/4}{\delta^2} \right) \right] \\
&& + \frac{2t-1}{\delta} \left[ \sech^2  \left( \frac{t-1/4}{\delta^2} \right) - \sech^2 \left( \frac{t-3/4}{\delta^2} \right) \right]  \, , 
\end{eqnarray*}
while applying the shear conditions of Theorem~\ref{theorem:eigenvector} as in the previous example locally near $ (1,0) $ gives
$ c_2(x,t) = 0 $ and 
\[
c_1(x,t) = \left[ \delta 2 \pi^2 A \tanh \left( \frac{t-1/2}{\delta^2} \right) - \frac{1}{\delta} \sech^2 \left( \frac{t-1/2}{\delta^2} \right) \right] x_2
I_{(1,0)}(x_1,x_2) \, .
\]
Thus, a control velocity which simultaneously attempts to achieve the required hyperbolic trajectory and tangent vector
rotation can be constructed by summing these:
\begin{eqnarray}
\hspace*{-1.5cm} \left( \begin{array}{c}
 \! c_1(x_1,x_2,t) \! \\ \mbox{} \\ \! c_2(x_1,x_2,t) \! \end{array} \right) 
& \! =  \! & \delta \! \left( \begin{array}{c}
\left( 2 \! - \! 2 \pi^2 A t \! + \! \pi^2 A \right) \left[ \tanh \! \left( \frac{t\! -\!1/4}{\delta^2} \right) \!-\! \tanh \!
\left( \frac{t\!-\!3/4}{\delta^2} \right) \right] \!+\! 
2 \pi^2 A \tanh \!\left( \frac{t\!-\!1/2}{\delta^2} \right) x_2 I_{(1,0)}(x_1,x_2)
 \\ 0 \end{array} \right) \nonumber \\
 & &   +    \frac{1}{\delta}  \left( \begin{array}{c} 
 (2t\!-\!1) \left[ \sech^2 \! \left(  \frac{t\!-\!1/4}{\delta^2} \right) \!-\! \sech^2 \!\left( \frac{t\!-\!3/4}{\delta^2} \right) \right]
 \!-\! \sech^2 \!\left( \frac{t\!-\!1/2}{\delta^2} \right) x_2 I_{(1,0)}(x_1,x_2) 
 \\ 0
 \end{array} \right) \, .
\label{eq:tgabrupt}
\end{eqnarray}
The $ 1/\delta $ terms above represent the spikes (associated with the time-derivatives in (\ref{eq:abrupt})) needed;
these are \edit{smooth approximations} of Dirac impulses.  The system (\ref{eq:taylorgreencontrol}) was numerically
examined with $ c $ given by (\ref{eq:tgabrupt}), with the choice of parameters $ A = 1 $ and $ \delta = 0.1 $.
The  $ t $-discretisation of the previous example was used, with
the forward FTLE field computed at each time using the maximum available data (i.e., till $ T = 2 $).  By
evaluating the slopes and intercepts from the extracted ridge, the values of $ \theta_{\tn{s}}(t) $ and
the $ x_1 $-component of $ a(t) $ were respectively computed at values of $ t \in [0,1] $.  The results, in circles,
are compared with the required curves (\ref{eq:abruptsmooth}) in Fig.~\ref{fig:abrupt}.  There is some error
near the abrupt changes, which is inevitable since the FTLE ridges become ambiguous at discontinuities.
Nevertheless, the efficacy of the control strategy in this nearly discontinuous situation is remarkable.

\begin{figure}[t]
\includegraphics[width=0.4 \textwidth, height=0.3 \textheight]{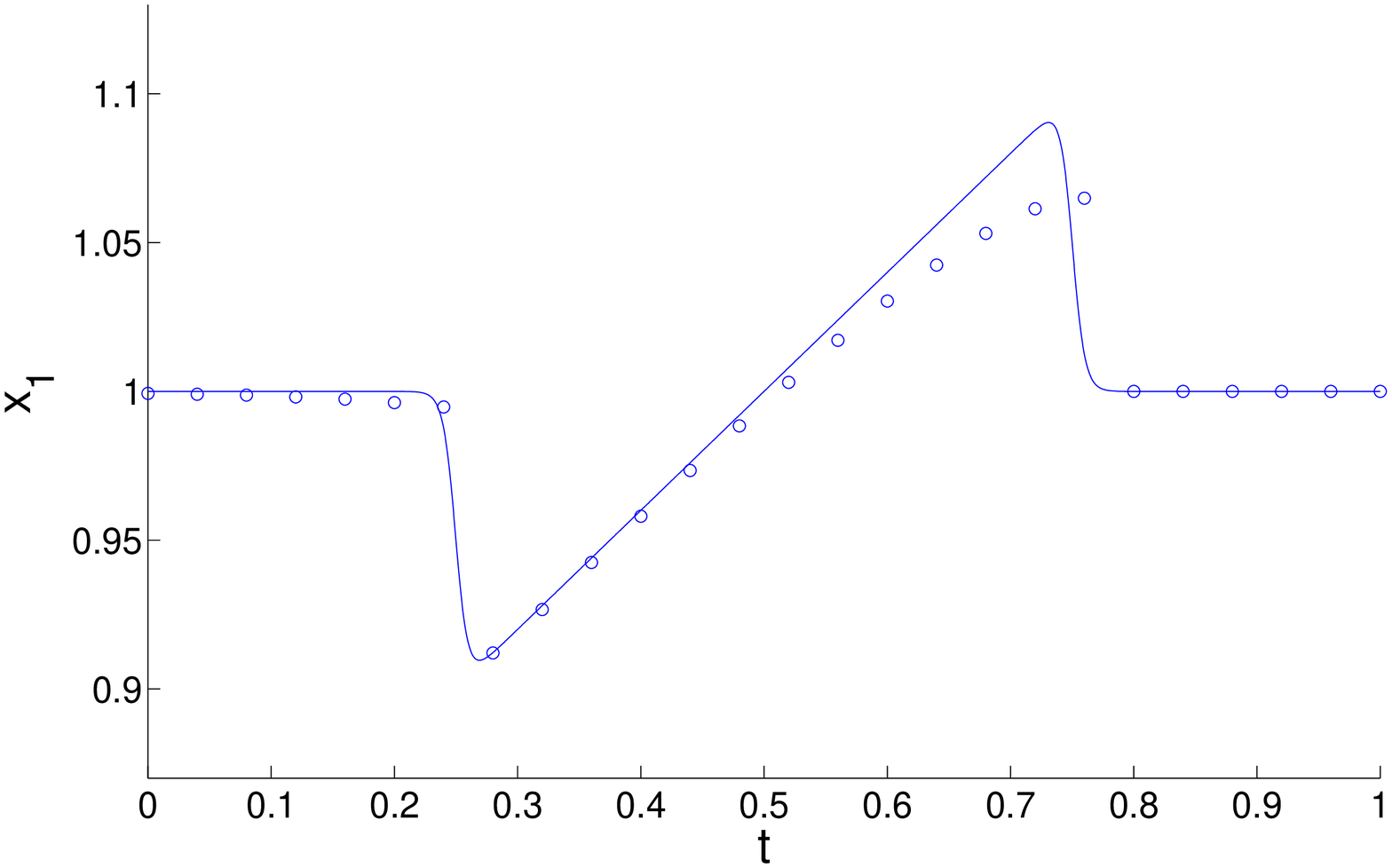}
\includegraphics[width=0.55 \textwidth, height=0.3 \textheight]{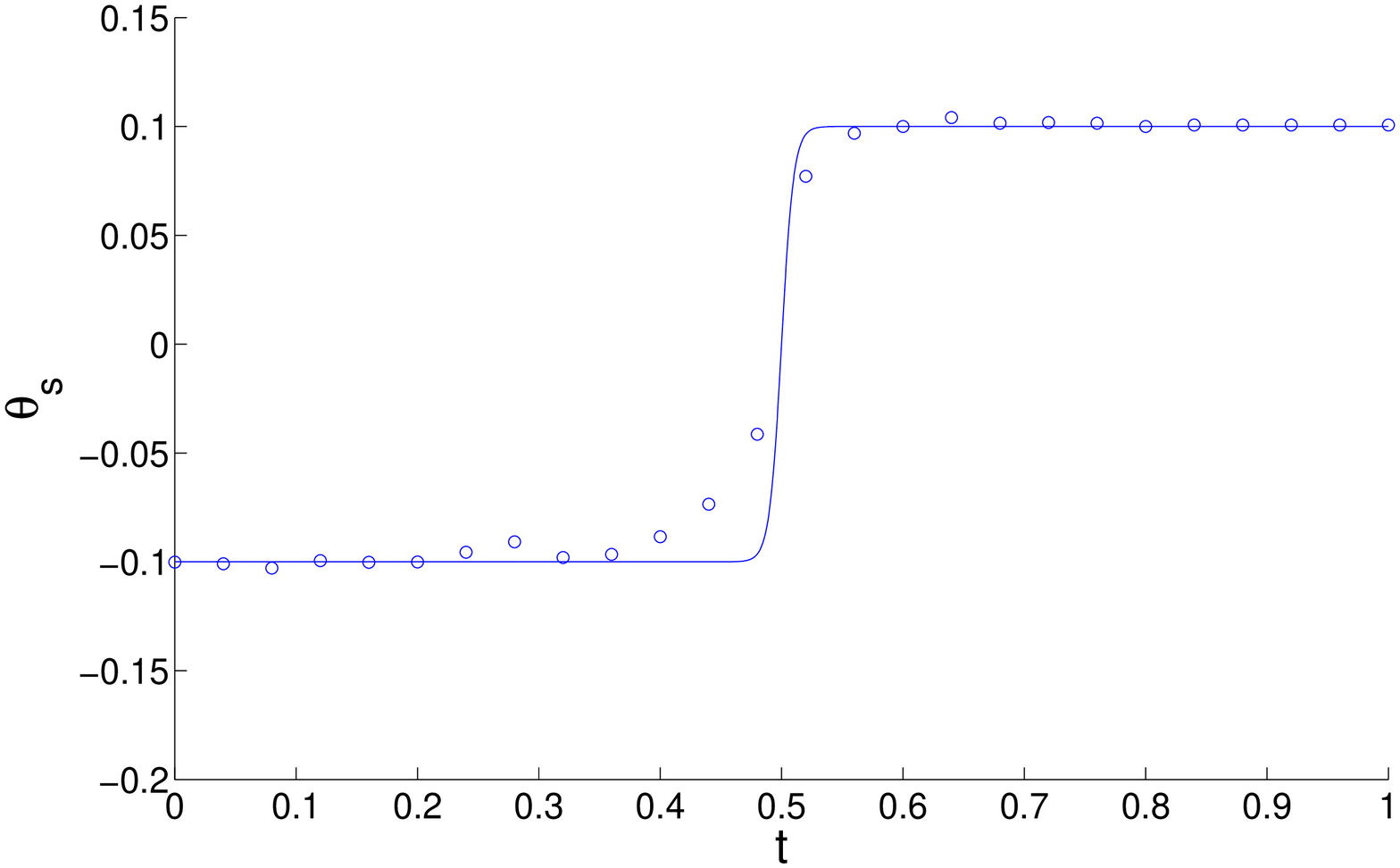}
\caption{Results for the control velocity (\ref{eq:tgabrupt}) with $ \delta = 0.1 $, $ A = 1 $,
and $ T = 2 $, extracted by computing the intercept and slope of the FTLE ridges using linear regression [circles]:
$ x_1 $-coordinate of hyperbolic trajectory (left panel) and $ \theta_{\tn{s}} $ (right panel) along with the
required curves (\ref{eq:abruptsmooth}) [solid].}
\label{fig:abrupt}
\end{figure}

\section{Concluding remarks}

This article approaches the issue of nonautonomous local tangents to
stable and unstable manifolds, from two perspectives.
  First, it hopes to address the finite-time situation in a
sense that would appear reasonable with data availability, while retaining the nonautonomous viewpoint
within this finite time-range.
Under the ansatz of the flow being nearly autonomous, leading-order approximations for \edit{hyperbolic trajectories
and their attached local stable/unstable manifolds} were stated.
The relationship between the \edit{local tangent vector} rotation and the nonautonomous velocity shear was quantified.    
An attempt to characterise how unknown data from outside a
finite-time interval affects flow entities was introduced, by investigating the dependence of an error which depends
on the finiteness parameter $ T $.
The nearly autonomous hypothesis is of course restrictive, and one potential extension would be to assume that the
flow is a perturbation of a nearby flow which, while not necessarily autonomous, has its relevant features (hyperbolic
trajectory, stable and unstable manifolds) {\em known}.  This idea was formally used recently \cite{controlnd} for
controlling high-dimensional hyperbolic trajectories.

Second, this article addresses the question of controlling the direction of
emanation of stable and unstable manifolds from a hyperbolic trajectory. Such local results influence the {\em global} stable and unstable manifolds, i.e., the global transport templates,
since exponential decay definitions \cite{coppel,battellilazzari,palmer} for global invariant manifolds depend on {\em local} decay.
Standard approaches for defining global stable and unstable manifolds \cite{guckenheimerholmes,arrowsmithplace} indeed
depend on first defining local stable/unstable manifolds, and then obtaining the global manifolds by `flowing' these in time.
The local shear velocity required for a given nonautonomous motion of these directions was obtained in Theorem~\ref{theorem:control}.  The finite-time version
of these (see Remark~\ref{remark:control}) simply uses the full data available in the time-range $ [-T,T] $ 
in doing the computation of the shear requirement.  The efficacy of this finite-time process was numerically demonstrated
using both a time-periodic and a time-discontinuous specification of the tangent vector directions, the latter situation 
pushing the boundaries of the theorems.  To achieve genuinely time-discontinous flow trajectories, impulsive terms are required in the velocities.  It has been shown that in the presence of impulsive vector fields
one needs to think of stable and unstable {\em pseudo}-manifolds which reset themselves at the time instance at which 
there is an impulse \cite{aperiodic,impulsive}.
For the purposes of the numerical verifications of this article, marginally \edit{smooth approximations}
of the discontinuous functions were employed.  In both the time-periodic and time-impulsive implementations of the
control strategy, excellent performance (as measured by the rotation of FTLE ridges with time) was obtained.  Any
specified (but small) time-varying reorientation of the stable and unstable manifold directions seems to be possible, 
providing a tool for controlling the essential skeleton of fluid flows.  In particular, by focussing energy on a {\em localised}
region near the time-varying hyperbolic trajectory, a global transport impact can be achieved.  Further analysis and
development of these idea, building also on \cite{saddle_control,controlnd,manifold_control}, are underway.

\appendix
\section{Proof of Theorem~\ref{theorem:eigenvector} (Local manifold directions)}
\label{app:eigenvector}

The system (\ref{eq:flow}) is, under the conditions of Hypothesis~\ref{hyp:F}, equivalent to
\begin{equation}
\dot{x} = f(x) + \left[ F(x,t,\eps) - F(x,t,0) \right] =: f(x) + \eps h(x,t) + {\mathcal O}(\eps^2)
\label{eq:fg}
\end{equation}
where the higher-order term in uniformly bounded, and $ g(x,t) = \eps h(x,t) + {\mathcal O}(\eps^2) $.  Let $ \bar{x}^{\tn{s}}(p) $ be a solution to (\ref{eq:fg}) when 
$ \eps = 0 $ such that $ \bar{x}^{\tn{s}}(p) \rightarrow a $ as $ p \rightarrow \infty $; this solution can be used to
parametrise a branch of $ a $'s stable manifold for $ p \in [P,\infty] $, for $ P $ as negative as required.
\edit{Balasuriya \cite{tangential,open} establishes a parametrisation of the perturbed stable manifold in his Theorems~2.7 and 2.8 \cite{tangential}.  These results shall be recast for the present context as
\begin{equation}
x^{\tn{s}}(p,t) = \bar{x}^{\tn{s}}(p) + \eps \frac{ M^{\tn{s}}(p,t)}{\left| f \left( \bar{x}^{\tn{s}}(p) \right) \right|^2} 
f^\perp \left( \bar{x}^{\tn{s}}(p) \right) + J^{\tn{n}} (p,t,\eps)
\frac{f^\perp \left( \bar{x}^{\tn{s}}(p) \right)}{\left| f \left( \bar{x}^{\tn{s}}(p) \right) \right|} 
+ J^{\tn{t}}(p,t,\eps) 
\frac{f \left( \bar{x}^{\tn{s}}(p) \right)}{\left| f \left( \bar{x}^{\tn{s}}(p) \right) \right|} \, , 
\label{eq:errorsplit}
\end{equation}
where the superscripts for $ J $ are for the {\em n}ormal and {\em t}angential components respectively.  The $ {\mathcal O}(\eps) $
term of the normal component is expressed in terms of the Melnikov function
\begin{equation}
M^{\tn{s}}(p,t) := - \int_t^\infty \exp \left[ \int_{\tau-t+p}^p {\mathrm{Tr}} \, D  f \left( \bar{x}^{\tn{s}}(\xi) \right) \d \xi \right] f^\perp \left( 
\bar{x}^{\tn{s}}(\tau-t+p) \right)^{\edit{\top}} h \left( \bar{x}^{\tn{s}}(\tau-t+p), \tau \right) \, \d \tau \, .
\label{eq:ms}
\end{equation}
It is possible \cite{tangential} to similarly write the tangential component in terms of a known expression for its
$ {\mathcal O}(\eps) $-term and a higher-order error term, but as will be seen, this precise expression will not be necessary 
for this proof.
}
Now, as $ p \rightarrow \infty $, $ x^{\tn{s}}(p,t) \rightarrow a(t) $, the hyperbolic trajectory, along the stable manifold
direction, in each fixed time-slice $ t $.  Thus, the tangent vector to this, at the point $ a(t) $, can be obtained by
applying the limit $ p \rightarrow \infty $ to the $ p $-derivative of $ x^{\tn{s}}(p,t) $.  This is 
\begin{eqnarray}
x_p^{\tn{s}} & =  & \bar{x}_p^{\tn{s}} + \eps \left[ \frac{M^{\tn{s}} f_p^\perp}{\left| f \right|^2} + \left( \frac{M_p^{\tn{s}}}{\left| f \right|^2} - 
\frac{2 M^{\tn{s}} f^{\edit{\top}} f_p}{\left| f \right|^4} \right) f^\perp \right]
\edit{+  \left[ \frac{J^{\tn{n}} f_p^\perp + J_p^{\tn{n}} f^\perp}{\left| f \right|} - \frac{J^{\tn{n}} f^\perp f^\top f_p}{\left| f\right|^3} \right] }
\nonumber \\
& & \edit{+  \left[ \frac{J^{\tn{t}} f_p + J_p^{\tn{t}} f}{\left| f \right|} - \frac{J^{\tn{t}} f f^\top f_p}{\left| f\right|^3} \right] }
\label{eq:xp}
\, , 
\end{eqnarray}
where the $ p $-subscript represents the partial derivative, and the arguments $ (p,t,\eps) $ for $ M^{\tn{s}} $ \edit{and
$ J^{\tn{n,t}} $}, and the argument $ \bar{x}^{\tn{s}}(p) $ for $ f $ have been suppressed for brevity.  Since the $ p \rightarrow \infty $
limit is required, in this limit
\begin{equation}
\bar{x}^{\tn{s}}(p) \sim a + c v_{\tn{s}} e^{\lambda_{\tn{s}} p}
\label{eq:barx}
\end{equation}
for a constant $ c \ne 0 $ can be applied; this is since the linearised flow $ \dot{y} = \left( D f \right) y $ dominates near $ a $, and
$ \bar{x}^{\tn{s}} $ specifically comes in along the stable manifold (tangential to $ v_{\tn{s}} $ with decay rate $ \lambda_{\tn{s}} $)
in this limit.  The $ c $ represents a choice of `initial condition' along the stable manifold, and as will be clear, is
inconsequential in the final result.
Now, $ f \left( \bar{x}^{\tn{s}}(p) \right) = \bar{x}_p^s(p) $ since $ \bar{x}^{\tn{s}}(p) $ is a solution to (\ref{eq:fg}) 
when $ \eps = 0 $, and thus
\begin{equation}
f \left( \bar{x}^{\tn{s}}(p) \right) \sim c \lambda_{\tn{s}} v_{\tn{s}} e^{\lambda_{\tn{s}} p} \quad , \quad 
f^\perp \left( \bar{x}^{\tn{s}}(p) \right) \sim c \lambda_{\tn{s}} v_{\tn{s}}^\perp e^{\lambda_{\tn{s}} p} \quad {\mathrm{and}} \quad
f_p \left( \bar{x}^{\tn{s}}(p) \right) \sim c \lambda_{\tn{s}}^2 v_{\tn{s}} e^{\lambda_{\tn{s}} p} \, .
\label{eq:flargep}
\end{equation}
\edit{Of the three bracketed terms in (\ref{eq:xp}), the first two are therefore vectors in the $ v_{\tn{s}}^\perp $ direction,
whereas the third is in the $ v_{\tn{s}} $ direction.  The second term in the limit $ p \rightarrow \infty $ behaves according to
\[
 \left[ \frac{J^{\tn{n}} f_p^\perp + J_p^{\tn{n}} f^\perp}{\left| f \right|} - \frac{J^{\tn{n}} f^\perp f^\top f_p}{\left| f\right|^3} \right] 
 \rightarrow \left[ J^{\tn{n}} {\mathrm{sign}} \left( c \right) \left| \lambda_{\tn{s}} \right| + J_p^{\tn{n}} {\mathrm{sign}} \left( c \lambda_{\tn{s}} \right) - J^{\tn{n}} {\mathrm{sign}} \left( c \right) \left| \lambda_{\tn{s}} \right| 
 \right] v_{\tn{s}}^\perp \, .
 \]
It is shown in Lemma~\ref{lemma:error} in Appendix~\ref{app:error}
that as $ p \rightarrow \infty $,  $ J^{\tn{n}}(p,t,\eps) $ in this
limit remains uniformly $ {\mathcal O}(\eps^2) $ for $ (t,\eps) \in \R \times [0,\eps_0) $. 
Thus, the error in discarding
 this term is $ {\mathcal O}(\eps^2) $ in the normal direction.  The third term of (\ref{eq:xp}) contains 
 the tangential projection $ J^{\tn{t}} $ and
 its $ p $-derivative, and in this case it will turn out that it is only required to show that this remains $ {\mathcal O}(\eps) $.
 This is easiest accomplished by analysing (\ref{eq:errorsplit}), from which
 \[
 J^{\tn{t}}(p,t,\eps) = \frac{f^\top \left( \bar{x}^{\tn{s}}(p) \right)}{\left| f \left( \bar{x}^{\tn{s}}(p) \right) \right|} \left[
 x^{\tn{s}}(p,t) - \bar{x}^{\tn{s}}(p) \right] \rightarrow v_{\tn{s}}^\top \left[ a(t) - a \right] \, .
 \]
 Since $ a(t) - a $ remains uniformly $ {\mathcal O}(\eps) $, this asserts the existence of a constant $ K_1 $ such that
 $ \lim_{p \rightarrow \infty} \left| J^{\tn{t}}(p,t,\eps) \right| \le \eps K_1 $.  Moreover, 
 \[
 J_p^{\tn{t}}(p,t,\eps) = \left( 
 \frac{f^\top}{\left| f \right|} \right)_p \left[ x^{\tn{s}} - \bar{x}^{\tn{s}} \right]
 + \frac{f^\top}{\left| f \right|} \left[ x_p^{\tn{s}} - \bar{x}_p^{\tn{s}} \right]   
 \]
  in which the first term remains $ {\mathcal O}(\eps) $ by the same argument.  The second term represents
  the difference between the tangent vector directions of the perturbed and unperturbed stable manifolds in the limit
  of approaching the hyperbolic trajectory, and is thus also uniformly $ {\mathcal O}(\eps) $ by persistence of
  invariant manifolds \cite{yi,yistability}.  Therefore, the final bracketed term in (\ref{eq:xp}) is bounded by a term
$ K_2 \eps $, where $ K_2 $ is independent of $ (t,\eps) $.
}

\edit{Collecting all this information together, and} substituting the large $ p $ values into the expression for $ x_p^s $ yields
\[
x_p^{\tn{s}} = c \lambda_{\tn{s}} e^{\lambda_{\tn{s}} p} v_{\tn{s}} + \frac{\eps}{c \lambda_{\tn{s}} e^{\lambda_{\tn{s}} p}} 
\left( M_p^{\tn{s}} - \lambda_{\tn{s}} M^{\tn{s}} \right) v_{\tn{s}}^\perp  \edit{ +
{\mathcal O}(\eps^2) v_{\tn{s}}^\perp + {\mathcal O}(\eps) v_{\tn{s}} }\, .
\]
The rotational angle $ \theta_{\tn{s}} $ from $ v_{\tn{s}} $ towards $ v_{\tn{s}}^\perp $ is $ {\mathcal O}(\eps) $ and
thus equal to $ \tan \theta_{\tn{s}} $ to leading-order.    This is essentially the slope of the above tangent line
in an axis system $ \left( v_{\tn{s}} , v_{\tn{s}}^\perp\right) $.  Thus, 
\begin{equation}
\theta_{\tn{s}} =  \frac{ \frac{\eps}{c \lambda_{\tn{s}} e^{\lambda_{\tn{s}} p}} \left( M_p^{\tn{s}} - \lambda_{\tn{s}} M^{\tn{s}} \right) + {\mathcal O}(\eps^2)}{
 c \lambda_{\tn{s}} e^{\lambda_{\tn{s}} p} + {\mathcal O}(\eps) } 
=  \eps \frac{ M_p^{\tn{s}} - \lambda_{\tn{s}} M^{\tn{s}}}{c^2 \lambda_{\tn{s}}^2 e^{2 \lambda_{\tn{s}} p}} + 
\edit{\eps^2 \tilde{E}_{\tn{s}}(t,\eps)}
\label{eq:thetatemp}
\end{equation}
as $ p \rightarrow \infty $,  \edit{where $ \tilde{E}_{\tn{s}}(t,\eps) $ is uniformly bounded for $ (t,\eps) $.}   The $ p $-derivative of (\ref{eq:ms}) is now required in the limit $ p \rightarrow \infty $.  In
this limit, $ {\mathrm{Tr}} \, D  f \left( \bar{x}^{\tn{s}}(p) \right) \rightarrow \lambda_{\tn{s}} + \lambda_{\tn{u}} $, the sum of $ D f  $'s eigenvalues at $ a $.  Putting this along with the other large $ p $ estimates 
in (\ref{eq:flargep}) into (\ref{eq:ms}) gives the large $ p $ estimate
\[
M^{\tn{s}}(p,t) = - c \lambda_{\tn{s}} e^{\lambda_{\tn{u}} t} e^{\lambda_{\tn{s}} p} \int_t^\infty e^{\lambda_{\tn{u}} \tau} h^{\edit{\top}} \left( a + c v_{\tn{s}}
e^{\lambda_{\tn{s}}(\tau-t+p)}, \tau \right) \edit{v_{\tn{s}}^\perp} \, \d \tau \, .
\]
When computing $ M_p^{\tn{s}} - \lambda_{\tn{s}} M^{\tn{s}} $, the fact that $ M^{\tn{s}} $ is a product of $ e^{\lambda_{\tn{s}} p} $ with another
function of $ p $ leads to cancellations, and results in
\begin{eqnarray*}
M_p^{\tn{s}} - \lambda_{\tn{s}} M^{\tn{s}} & = & - c \lambda_{\tn{s}} e^{\lambda_{\tn{u}} t} e^{\lambda_{\tn{s}} p} \int_t^\infty e^{-\lambda_{\tn{u}} \tau} \left( c \lambda_{\tn{s}} v_{\tn{s}} e^{\lambda_{\tn{s}}(\tau-t+p)} \right)^{\edit{\top}} \edit{D} \left[ h^{\edit{\top}}
\left( a + c v_{\tn{s}} e^{\lambda_{\tn{s}}(\tau-t+p)}, \tau \right) \edit{v_{\tn{s}}^\perp} \right] \, \d \tau \\
& = & - c^2 \lambda_{\tn{s}}^2 e^{(\lambda_{\tn{u}}-\lambda_{\tn{s}})t} e^{2 \lambda_{\tn{s}} p} \int_t^\infty
e^{(\lambda_{\tn{s}} - \lambda_{\tn{u}})\tau} v_{\tn{s}}^{\edit{\top}} \edit{D} \left[ h^{\edit{\top}} \left( a + c v_{\tn{s}} e^{\lambda_{\tn{s}}(\tau-t+p)}, \tau \right) \edit{v_{\tn{s}}^\perp} \right] \d \tau \, .
\end{eqnarray*}
Substituting into (\ref{eq:thetatemp}), putting in $ g = \eps h + {\mathcal O}(\eps^2) $,  and applying $ p \rightarrow \infty $ gives (\ref{eq:slopes}), where the $ {\mathcal O}(\eps^2) $ terms have been combined into one term $ \eps^2 E_{\tn{s}}(t,\eps) $
with bounded $ E_{\tn{s}}(t,\eps) $ for $ t \in \edit{\R} $.  

Now, (\ref{eq:slopeu}) is similarly derived, by using the unstable manifold formulation (Theorems~2.1 and 2.3) due to
Balasuriya \cite{tangential}.   There is no
substantive difference in the derivation strategy, which shall be skipped for brevity. \hfill $ \Box $

\section{\edit{Proof of normal error term being uniformly $ {\mathcal O}(\eps^2) $}}
\label{app:error}

\edit{
\begin{lemma}
\label{lemma:error}
Under the conditions of Theorem~\ref{theorem:eigenvector}, then there exists $ K $ such that 
 error term $ J^{\tn{n}}(p, t,\eps) $ in (\ref{eq:errorsplit})
satisfies 
\[
\limsup_{p \rightarrow \infty} \left\{ \left| J^{\tn{n}}(p, t,\eps) \right| + \left| \frac{\partial J^{\tn{n}}}{\partial p}(p,t,\eps)
\right| \right\} \le \eps^2 K \quad {\mathrm{for}} \, (t,\eps) \in \R \times [0,\eps_0) \, .
\]
\end{lemma}
}
\begin{proof}
\edit{
For any $ \tau \in [t,\infty) $, define $ x_1^{\tn{s}} $ through
\begin{equation}
x^{\tn{s}}(p,\tau) = \bar{x}^{\tn{s}}(\tau-t+p) + \eps x_1^{\tn{s}}(p,\tau,\eps) \, ,
\label{eq:x1def}
\end{equation}
Equation~(3.1) by
Balasuriya \cite{tangential}, developed for the unstable manifold in that case, can be adapted to quantify this. 
The error required is obtained by
replacing the unstable manifold with the stable one, dividing by $ \left| f \left( \bar{x}^{\tn{s}}(p) \right) \right| $,  and 
then integating from $ t $ to $ \infty $.  In the present context, this translates to
\begin{equation}
 J^{\tn{n}}(p,t,\eps) = \eps^2 \int_t^\infty \frac{f^\perp \left( \bar{x}^{\tn{s}}(\tau - t  + p) \right)^\top}{ 
 \left| f \left( \bar{x}^{\tn{s}}(p) \right) \right|}  P(p,\tau,\eps) \,  \d \tau \, , 
\label{eq:errorperp}
\end{equation}
where
\begin{equation}
P(p,\tau,\eps) := \left[ \frac{1}{2} x_1^{\tn{s}}(p,\tau,\eps)^\top D^2 f \left( y_1(p,\tau) \right) +  D g \left( y_2(p,\tau),\tau,
 \eps \right) \right] x_1^{\tn{s}}(p,\tau,\eps )
\label{eq:P}
\end{equation}
in which the $ y_1 $ and $ y_2 $ are locations on the line segment between  $ \bar{x}^{\tn{s}}(\tau-t+p) $ and
$ x^{\tn{s}}(p,\tau) $ arising from Taylor's theorem applied to
$ f $ and $ g $.  The uniform boundedness of $ P $ as $ p \rightarrow \infty $
is first argued.  In this limit, $ x^{\tn{s}}(p,\tau) \rightarrow a(t) $ and $ \bar{x}^{\tn{s}}(\tau-t+p) \rightarrow 
a $, with $ \left| a - a(t) \right| $ remaining uniformly $ {\mathcal O}(\eps) $ in $ t $.  Thus, $ \limsup_{p \rightarrow \infty} 
x_1^{\tn{s}}(p,\tau,\eps) $ remains bounded uniformly.  Moreover, the hypotheses on $ f $ and $ g $ ensure that the
$ D^2 f $ and $ D g $ terms are also uniformly bounded in $ \tau $ under this limit.  Now, applying the large $ p $ estimates 
$ f \left( \bar{x}^{\tn{s}}(p) \right) \sim c \lambda_{\tn{s}} v_{\tn{s}} e^{\lambda_{\tn{s}} p} $  to (\ref{eq:errorperp}), one gets
\begin{eqnarray}
\left| \limsup_{p \rightarrow \infty} 
 J^{\tn{s}}(p,t,\eps)  \right| & \le & \eps^2 K_1 \left| \int_t^\infty  \limsup_{p \rightarrow \infty} 
 \frac{  \lambda_{\tn{s}} e^{\lambda_{\tn{s}}(\tau - t + p)} }{\lambda_{\tn{s}} e^{\lambda_{\tn{s}} p}} \left( K_2^\top v_{\tn{s}}^\perp \right) \d \tau \right| \nonumber  \\
& \le & \eps^2 K_3 e^{- \lambda_{\tn{s}} t} \frac{ e^{\lambda_{\tn{s}} \tau} }{\left| \lambda_{\tn{s}} \right|} \Big]_t^\infty =
\eps^2 K_4 \, .
\label{eq:errorperpbound}
\end{eqnarray} 
In the above, the dominated convergence theorem allowed the limit to be moved into the integral, and $ K_i $s will be
used throughout this proof to indicate constants (scalar or vector, dependending on context) independent of $ (p,t,\eps) $.
For the proof of Theorem~\ref{theorem:eigenvector} as presented in 
Appendix~\ref{app:eigenvector}, it is not just $ J^{\tn{s}} $ but its $ p $-derivative which needs to be addressed.  
From (\ref{eq:errorperp}), 
\begin{eqnarray*}
\frac{\partial}{\partial p} \left[ 
 J^{\tn{n}}(p,\tau,\eps) \right] & = & \eps^2 \int_t^\infty \frac{\partial}{\partial p} \left[ \frac{f^\perp \left( \bar{x}^{\tn{s}}(\tau - t  + p) \right)^\top}{ 
 \left| f \left( \bar{x}^{\tn{s}}(p) \right) \right|}  \right] P(p,\tau,\eps) \,  \d \tau \, \\
 & & + \eps^2 \int_t^\infty \frac{f^\perp \left( \bar{x}^{\tn{s}}(\tau - t  + p) \right)^\top}{ 
 \left| f \left( \bar{x}^{\tn{s}}(p) \right) \right|}  \frac{\partial P}{\partial p}(p,\tau,\eps) \,  \d \tau \, .
 \end{eqnarray*}
 Now, the $ p $-derivative of $ P $ remains bounded uniformly as $ p \rightarrow \infty $ since $ D^3 f $ and $ D^2 g $ 
 are uniformly bounded by hypothesis.  Therefore the second integral above is uniformly bounded by exactly the same
 argument already made for the (undifferentiated) $ J^{\tn{n}} $.  In the first integral for large $ p $, the fact that the term
 whose $ p $ derivative is to be taken collapses to $ v_{\tn{s}}^\perp e^{\lambda_{\tn{s}}(\tau - t)} $ has already been 
 established in (\ref{eq:errorperpbound}).  Thus, that integral contibutes zero.  This establishes that the component of
 $ \partial J^{\tn{s}}(p,t,\eps) / \partial p $ in the normal direction  is uniformly $ {\mathcal O}(\eps^2) $ for
 $ (t,\eps) \in \R \times [0,\eps_0) $, thereby completing the proof of Lemma~\ref{lemma:error}.
}
\end{proof}


{\bf Acknowledgements:}
Support from Australian Research Council  grant FT130100484, conversations with Gary Froyland, \edit{and
critical feedback from an anonymous referee}, are 
gratefully acknowledged.



\end{document}